\documentclass[a4paper,12pt,leqno]{amsart}
\usepackage{amsmath,amsthm,amssymb,latexsym,epsfig,graphicx,subfigure}
\numberwithin{equation}{section} \numberwithin{figure}{section}

\newtheorem{thm}{Theorem}[section]

\newtheorem{lm}[thm]{Lemma}

\theoremstyle{definition}
\newtheorem{notation}[thm]{Notation}
\newtheorem{pr}[thm]{Proposition}
\theoremstyle{definition}
\newtheorem{df}[thm]{Definition}
\theoremstyle{definition}
\newtheorem{rem}[thm]{Remark}

\newcommand{\Rn}{\mathbb{R}^{n}}
\newcommand{\hn}{\mathbb{R}^{n-1}}
\newcommand{\R}{\mathbb{R}}

\newcommand{\N}{\mathbb{N}}

\renewcommand{\hn}{\mathbb{H}^n}
\newcommand{\C}{\mathbb{C}}

\newcommand{\ha}{\mathcal{H}}
\newcommand{\h}{\mathbb{H}}

\renewcommand{\ss}{\mathcal{S}}

\newcommand{\s}{\sigma}
\newcommand{\cu}{\overline{U}}

\newcommand{\stm}{\setminus}
\renewcommand{\o}{\Omega}
\newcommand{\D}{\Delta}

\newcommand{\anah}{\nabla_\h}
\newcommand{\divh}{\di_\h}
\newcommand{\g}{\gamma}
\newcommand{\G}{\Gamma}
\renewcommand{\dh}{\Delta_\h}
\newcommand{\de}{\delta}

\newcommand {\grtrsim} {\ {\raise-.5ex\hbox{$\buildrel>\over\sim$}}\ }
\newcommand{\e}{\varepsilon}

\newcommand{\ra}{\rightarrow}
\newcommand{\khii}{\text{\lower -.4ex\hbox{$\chi$}}}
\DeclareMathOperator{\spt}{spt} \DeclareMathOperator{\dist}{dist}
 
 \DeclareMathOperator{\diam}{d}
 \DeclareMathOperator{\di}{div}

\renewcommand{\t}{\tau}
\newcommand{\f}{\varphi}
\renewcommand{\a}{\alpha}

\begin{document}
\title [singular integrals, self-similar sets and removability in $\hn$]{singular integrals on self-similar sets and removability for lipschitz harmonic functions in Heisenberg groups}
\author{Vasilis Chousionis and Pertti Mattila}

\thanks{Both authors were supported by the Academy of Finland.} 

\address{Vasilis Chousionis.  Departament de
Ma\-te\-m\`a\-ti\-ques, Universitat Aut\`onoma de Bar\-ce\-lo\-na, Spain}

\email{chousionis@mat.uab.cat}

\address{Pertti Mattila, Department of Mathematics and Statistics,
P.O. Box 68,  FI-00014 University of Helsinki, Finland}

\email{pertti.mattila@helsinki.fi}

\begin{abstract}
In this paper we study singular integrals on small (that is, measure zero and lower than full dimensional) subsets of 
metric groups. The main examples of the groups we have in mind are Euclidean spaces and Heisenberg groups. In addition 
to obtaining results in a very general setting, the purpose 
of this work is twofold; we shall extend some results in Euclidean spaces to more general kernels than previously 
considered, and we shall obtain in Heisenberg groups some applications to harmonic (in the Heisenberg sense) functions of 
some results known earlier in Euclidean spaces.

\end{abstract}

\maketitle

\section{Introduction}

The Cauchy singular integral operator on one-dimensional subsets of the complex plane has been studied extensively for 
a long time with many applications to analytic functions, in particular to analytic capacity and removable sets of 
bounded analytic functions. There have also been many investigations of the same kind for the Riesz singular 
integral operators with the kernel $x/|x|^{-m-1}$ on $m$-dimensional subsets of $\R^n$. One of the general themes has 
been that boundedness properties of these singular integral operators imply some geometric regularity properties of the 
underlying sets, see, e.g., \cite{DS}, \cite{M}, \cite{M4}, \cite{Pa}, \cite{tre} and \cite{mro}. Standard self-similar Cantor sets have often served as examples where such results were first 
established. This tradition was started by Garnett in \cite{g} and Ivanov in \cite{i} who used them as examples of removable sets for bounded analytic functions with positive length. 
Later studies of such sets include \cite{gb}, \cite{j}, \cite{jm}, \cite{i2}, \cite{M2}, \cite{mtv}, \cite{mtv2} and \cite{gv} in connection with the Cauchy integral in the complex plane, \cite{mt} and \cite{t2} in connection with the Riesz transforms in higher dimensions, and \cite{dc} and \cite{C} in connection with other kernels. In this paper we shall first derive criteria for the unboundedness 
of very general singular integral operators on self-similar subsets of metric groups with dilations and then 
give explicit examples in Euclidean spaces and Heisenberg groups on which these criteria can be checked. 

Today quite complete results are known for the Cauchy integral and for the removable sets of bounded analytic functions. 
The new progress started from Melnikov's discovery in \cite{me} of the relation of the Cauchy kernel to the so-called
Menger curvature. This relation was applied by Melnikov and Verdera in \cite{mv} to obtain a simple proof of the 
$L^2$-boundedness of the Cauchy singular integral operator on Lipschitz graphs, and in \cite{MMV} in order to 
get geometric characterizations of those Ahlfors-David regular sets on which the Cauchy singular integral operator
is $L^2$-bounded and of those which are removable for bounded analytic functions. This progress culminated in 
David's characterization in \cite{d} of removable sets of bounded analytic functions among sets of finite length as those 
which intersect every rectifiable curve in zero length, and in Tolsa's complete Menger curvature integral 
characaterization in \cite{To} of all removable sets of bounded analytic functions. Much less known is known in higher 
dimensions for the Riesz transforms and removable sets for Lipschitz harmonic functions, for some results, 
see e.g. \cite{mpa}, \cite{M4}, \cite{V}, \cite{l}, \cite{vo}, \cite{tr}, \cite{t2} and \cite{env}. There are various reasons why the 
Lipschitz harmonic functions are a natural class to study, one of them is that by Tolsa's result in the plane 
the removable sets for bounded analytic and Lipschitz harmonic functions are exactly the same. Also the analog 
for the Lipschitz harmonic functions of the above mentioned David's result for sets of finite length was first proved in 
\cite{dm}.   

In \cite{CM} analogs of the results in \cite{mpa} and \cite{M4} were proven in Heisenberg groups for Riesz-type kernels.
They imply in particular that the operators are unbounded on many self-similar fractals. 
An unsatisfactory feature is that these kernels are not related to any natural function classes in Heisenberg groups 
in the same way as the Riesz kernels are related to harmonic functions in $\Rn$. This is one of the main reasons why we wanted 
to study more general kernels in this paper. Our kernels are now such that they include the (horizontal) gradient 
of the fundamental solution of the sub-Laplacian (or Kohn-Laplacian) operator which is exactly what is needed for the 
applications to the related harmonic functions. For many other recent developments on potential theory related to 
sub-Laplacians, see \cite{blu} and the references given there. 

We shall now give a brief description of the main results of the paper. In Section 2 we study a general metric group 
$G$ which is equipped with natural dilations $\delta_r:G \ra G, r>0$. All Carnot groups are such. For a kernel 
$K: G\times G \setminus \{(x,y):x=y\}\ra \R$  and a finite Borel measure $\mu$ on $G$ the maximal singular integral 
operator $T_K^*$ is defined by 
$$T_K^*(f)(x)=\sup_{\e>0}\left|\int_{G \setminus B(x,\e)} K(x,y)f(y)d\mu y\right|.$$
Suppose that $C=\bigcup_{i=1}^N S_i (C)$ is a self-similar Cantor set generated by the similarities 
$S_i=\t_{q_i}\circ \delta_{r_i}, i=1,\dots,N$, where $\t_{q_i}$ is the left translation by $q_i \in G$ and $r_i \in (0,1)$. 
Let $s>0$ be the Hausdorff dimension of $C$ and suppose that the kernel 
$K:=K_\o$ is $s$-homogeneous:
$$K_\o (x,y)=\frac{\o(x^{-1}\cdot y)}{d(x,y)^s},\ x,y \in G \setminus \{(x,y):x=y\},$$
where $\o: G \ra \R$  is a continuous and homogeneous function of degree zero, that is, 
$$\o(\delta_r(x))=\o(x)\ \text{for all} \ x\in G,r>0.$$
In Theorem \ref{unb} we prove that if there exists a fixed point $x$ for some iterated map 
$S_w:=S_{i_1}\circ \dots \circ S_{i_n}$ such that 
$$\int_{C \setminus S_w(C)} K_\o (x,y) d \mathcal{H}^s y \neq 0$$
then $T^*_{K_\o}$ is unbounded in $L^2(\mathcal{H}^s \lfloor C)$, where $\mathcal{H}^s \lfloor C$ is the restriction of the $s$-dimensional Hausdorff measure $\ha^s$ to $C$. We shall give a simple example in the plane where this 
criterion can be applied.

In Section 3 we shall study in the Heisenberg group $\hn$ removable sets for the $\dh$-harmonic functions, 
that is, solutions of the sub-Laplacian 
equation $\dh f=0$. As in the classical case in $\R^n$, see \cite{ca}, the removable sets for the bounded $\dh$-harmonic functions 
can be characterized as polar sets or the null-sets of a capacity with the critical Hausdorff dimension $Q-2$ where 
$Q=2n+2$ is the Hausdorff dimension of $\hn$, see Remark 13.2.6 of \cite{blu}. We shall verify in Theorem \ref{pos} that the critical dimension for the Lipschitz $\dh$-harmonic functions is $Q-1$, in accordance with the classical case, by proving that for a compact subset $C$ of $\hn$, $C$ is removable, if $\mathcal{H}^{Q-1}(C)=0$, and $C$ is not removable,  
if the Hausdorff dimension $\dim C>Q-1$. For this and the later applications to self-similar sets, we need a representation theorem for Lipschitz functions which are $\dh$-harmonic outside a compact set $C$ with finite $(Q-1)$-dimensional Hausdorff measure. This is given in Theorem \ref{main} and it tells us that such a function can be written in a neighborhood of $C$ 
as a sum of a $\dh$-harmonic function and a potential whose kernel is the fundamental solution of $\dh$. Finally in Section 4 we present a family of self-similar Cantor sets with positive and finite $(Q-1)$-dimensional Hausdorff measure which are 
removable for Lipschitz $\dh$-harmonic functions.  

Throughout the paper we will denote by $A$ constants which
may change their values at different occurrences, while constants
with subscripts will retain their values. We remark that as usual the notation $x \lesssim y$ means $x \lesssim A y$ for some constant $A$ depending only on structural constants, that is, the dimension $n$, the regularity constants of certain measures and the constant arising from the global equivalence of the metrics $d$ and $d_c$ defined in Section 3.

We would like to thank the referee for many useful comments. 

\section{Singular Integrals on self-similar sets of metric groups}

Throughout the rest of this section we assume, as in \cite{ms}, that $(G,d)$ is a complete separable metric group with the following properties:
\begin{enumerate}
\item The left translations $\tau_q:G \ra G$, 
$$ \tau_q(x)= q\cdot x,x \in G,$$
are isometries for all $q \in G.$
\item There exist dilations $\delta_r:G \ra G, r>0,$
which are continuous group homomorphisms for which,
\begin{enumerate}
\item $\delta_1=$ identity,
\item $d(\delta_r(x),\delta_r(y))=rd(x,y)$ for $x,y \in G,r>0$,
\item $\delta_{rs}=\delta_r \circ \delta_s$.
\end{enumerate}
It follows that for all $r>0$, $\delta_r$ is a group isomorphism with $\delta_r^{-1}=\delta_{\frac{1}{r}}$.
\end{enumerate}
The closed and open balls with respect to $d$ will 
be denoted by $B(p,r)$ and $U(p,r)$. By $\diam(E)$ we will denote the diameter of $E \subset G$ with respect to the metric $d$.

We denote by $\mathcal{H}^s,s\geq 0,$ the $s$-dimensional Hausdorff measure obtained from the metric $d$, i.e. for $E \subset G$ and $\delta >0$, $\mathcal{H}^s (E)=\sup_{\delta>0} \mathcal{H}^s_\delta (E)$, where
$$\mathcal{H}^s_\delta(E)=\inf \left\{\sum_i  \diam(E_i)^s: E \subset \bigcup_i E_i,\diam (E_i)<\delta \right\}.$$
In the same manner the $s$-dimensional spherical Hausdorff measure for $E \subset G$ is defined as $\ss^s (E)=\sup_{\delta>0} \ss^s_\delta (E)$, where
$$\ss^s_\delta(E)=\inf \left\{\sum_i  r^s_i: E \subset \bigcup_i B(p_i,r_i),r_i \leq \delta,p_i \in G \right\}.$$
Translation invariance and homogeneity under dilations of the Hausdorff measures follows as usual, therefore for $A \subset G, \ p \in G$  and $s,r \geq 0$,
$$\mathcal{H}^s (\t_p(A))=\mathcal{H}^s(A) \ \text{and} \ \mathcal{H}^s (\delta_r(A))=r^s \mathcal{H}^s(A)$$  and the same relations hold for the spherical Hausdorff measures as well.

Let $\mu$ be a finite Borel measure on $G$ and let a Borel measurable $K: G\times G \setminus \{(x,y):x=y\}\ra \R$ be a kernel which is bounded away from the diagonal, i.e., $K$ is bounded in $\{(x,y):d(x,y)>\delta\}$ for all $\delta>0$.
The truncated singular integral operators associated to $\mu$ and $K$ are defined for $f\in L^1(\mu)$ and $\e>0$ as,
$$T_\e(f)(y)=\int_{G \setminus B(x,\e)} K(x,y)f(y)d\mu y,$$
and the maximal singular integral is defined as usual,
$$T_K^*(f)(x)=\sup_{\e>0}|T_\e(f)(x)|.$$

We are particularly interested in the following class of kernels.
\begin{df}
\label{hom} For $s>0$ the $s$-homogeneous kernels are of the form,
$$K_\o (x,y)=\frac{\o(x^{-1}\cdot y)}{d(x,y)^s},\ x,y \in G \setminus \{(x,y):x=y\},$$
where $\o: G \ra \R$  is a continuous and homogeneous function of degree zero, that is, 
$$\o(\delta_r(x))=\o(x)\ \text{for all} \ x\in G,r>0.$$
\end{df}

In the classical Euclidean setting homogeneous kernels have been studied widely, see e.g. \cite{Gr}. The Hilbert transform in the line, the Cauchy transform in the complex plane and the Riesz transforms in higher dimensions are the best known singular integrals associated to homogeneous kernels. In \cite{h} Huovinen studied general one-dimensional homogeneous kernels in the plane.  

In $\Rn$ the lower dimensional coordinate $s$-Riesz kernels,
$$R_s^i(x,y)=\frac{x_i-y_i}{|x-y|^{s+1}},s \in (0,n), i=1,\dots,n,$$
are often studied in conjunction with Ahlfors-David regular measures:
\begin{df}\label{AD}
A Borel measure $\mu$ on a metric space $X$ is Ahlfors-David regular, or 
AD-regular, if for some positive numbers $s$ and $A$,
$$r^s/A \leq \mu(B(x,r)) \leq Ar^s\ \text{for all}\ x\in \spt\mu, 0<r<\diam(\spt\mu),$$
where $\spt\mu$ stands for the support of $\mu$.
\end{df}

The related central open question in $\Rn$ asks if the $L^2(\mu)$-boundedness of the $s$-Riesz transforms, $s \in \N \cap [1,n)$, forces the support of $\mu$ to be $s$-uniformly rectifiable or even simply $s$-rectifiable. In the case of $s=1$ it was answered positively in \cite{MMV}, and it remains an open problem for $s>1$. It originates from the fundamental work of David and Semmes, see e.g. \cite{DS}, and it can be heuristically understood in the following sense: 

Does the $L^2(\mu)$-boundedness of Riesz transforms impose a certain geometric regularity on the support of $\mu$? 

In order to achieve a better understanding for this problem, it is very natural to examine the behavior of Riesz transforms on fractals like self-similar sets. This is because although geometrically irregular they retain some structure. It should be expected that $s$-Riesz transforms cannot be bounded on typical self similar sets. Indeed this is the case as follows by results proved in \cite{M4} and \cite{V}. In \cite{CM} it was shown that an analogous result holds true even in the setting of Heisenberg groups. On the other hand it is not known if singular integrals associated to more general $s$-homogeneous kernels are unbounded on $s$-dimensional self-similar sets. In this direction Theorem \ref{unb} provides one criterion for unboundedness for homogeneous singular integrals valid in the general setting of this section. 

Let $\mathcal{S}=\{S_1,\dots,S_N\},N \geq 2,$  be an iterated function system (IFS) of similarities of the
form
\begin{equation}
\label{gsim}S_i=\t_{q_i}\circ \delta_{r_i}
\end{equation}
where $q_i \in G,r_i \in (0,1)$ and $i=1,\dots,N$. The self-similar set $C$ is the invariant set
with respect to $\mathcal{S}$, that is, the unique non-empty compact set such that
$$C=\bigcup_{i=1}^N S_i (C).$$
The invariant set $C$ will be called a separated self-similar set whenever the sets $S_i (C)$ are pairwise disjoint for $i=1,\dots,N$. It follows by a general result of Schief in \cite{s} that separated Cantor sets satisfy
$$0<\mathcal{H}^d(C)< \infty \ \text{ for } \ \sum_{i=1}^N r_i^d=1,$$
and the measure $\mathcal{H}^d \lfloor C$ is $d$-AD regular.

We denote by $\mathcal{I}$ the set of all finite words $w=(i_1,\dots,i_n) \in \{1,\dots,N\}^n$ with $n \geq 0$. Given any word $w=(i_1,\dots,i_n) \in \mathcal{I}$ its length is denoted by $|w|=n$ and for $m \leq n$, $w|_{m}=(i_1,\dots,i_m)$. We also adopt the following conventions:
$$S_w:=S_{i_1}\circ \dots \circ S_{i_n}\ \text{and} \ C_w=S_w(C).$$ The fixed points of $\mathcal{S}$ are exactly those $x \in C$ such that $S_w(x)=x$ for some $w \in \mathcal{I}$. In this case
$$\{x\}= \bigcap_{k=1}^\infty S_{w^k} (C)$$
where $|w^k|=k|w|$ and $w^k=(i_1,\dots,i_n,\dots\dots,i_1,\dots,i_n).$

\begin{thm}
\label{unb}
Let $\mathcal{S}=\{S_1,\dots,S_N\}$ be an iterated function system in $G$ generating a separated $s$-dimensional self-similar set $C$ and let $K_\Omega$ be an $s$-homogeneous kernel. If there exists a fixed point $x$ for $\mathcal{S}$,
$$\{x\}= \bigcap_{k=1}^\infty S_{w^k} (C),$$
such that
$$\int_{C \setminus C_w} K_\o (x,y) d \mathcal{H}^s y \neq 0,$$
then the maximal operator $T^*_{K_\o}$ is unbounded in $L^2(\mathcal{H}^s \lfloor C)$, moreover $\|T^\ast_{K_\o}(1)\|_{L^\infty(\mathcal{H}^s \lfloor C)}=\infty$.

\end{thm}

\begin{proof} Let $\mu=\mathcal{H}^s \lfloor C$ which as explained earlier satisfies
$$r^s/A_\mu \leq \mu(B(x,r)) \leq A_\mu r^s\ \text{for all}\ x\in C, 0<r<\diam (C).$$
Without loss of generality we can assume that
$$\int _{C \setminus C_w} K_\o(x,y)d \mu y = \eta >0.$$

Notice that the homogeneity of $\o$ implies that for all $v \in \mathcal{I}$,
\begin{equation}
\label{homome}
\o (S_v (x)^{-1} \cdot S_v (y))=\o(\delta_{r_{i_1}\dots r_{i_{|v|}}}(x^{-1}\cdot y))=\o (x^{-1}\cdot y).
\end{equation}

Therefore for all $k \in \N$, after changing variables $y=S_{w^k}(z)$ and recalling that $S_{w^k}(x)=x$,
\begin{equation*}
\begin{split}
\int_{C_{w^k}\stm C_{w^{k+1}}} K_\o(x,y)d \mu y&=\int_{C_{w^k}\stm C_{w^{k+1}}} \frac{\o(x^{-1}\cdot y)}{d(x,y)^s} d\mu y\\
&=\int _{C \stm C_w} \frac{\o(x^{-1}\cdot S_{w^k}(z))}{d(x,S_{w^k}(z))^s} (r_{i_1}\dots r_{i_{|w|}})^{ks} d \mu z\\
&=\int _{C \stm C_w} \frac{\o(S_{w^k} (S_{w^k}^{-1}(x))^{-1}\cdot S_{w^k}(z))}{d(S_{w^k} (S_{w^k}^{-1}(x)),S_{w^k}(z))^s} (r_{i_1}\dots r_{i_{|w|}})^{ks} d \mu z \\
&=\int _{C \stm C_w} \frac{\o(S_{w^k}^{-1}(x)^{-1}\cdot z)}{d(S_{w^k}^{-1}(x),z)^s} d \mu z\\
&=\int _{C \stm C_w}\frac{\o(x^{-1}\cdot z)}{d(x,z)^s} d\mu z \\
&= \eta.
\end{split}
\end{equation*}

Let $M$ be an arbitrary big positive number and choose $m \in \N$ such that $m\eta >M$. Then
$$\int_{C\stm C_{w^{m}}} K_\o(x,y)d \mu y=\sum_{i=0}^{m-1}\int_{C_{w^{i}}\stm C_{w^{i+1}}} K_\o(x,y)d \mu y>M.$$
Therefore by the continuity of $K_\o$ away from the diagonal there exist $m,m' \in \N, m<m',$ such that
\begin{equation}
\label{mest}
\int_{C\stm C_{w^{m}}} K_\o(p,y)d \mu y >M \ \text{for all}\ p \in C_{w^{m'}}.
\end{equation}
To simplify notation let $C^1=C_{w^m}$ and $C^2=C_{w^{m'}}$. Then
$$C\stm C^2 =\bigcup_{i=1}^{j_2} C_{2,i}$$
where the $C_{2,i}$'s are cylinder sets belonging to the same generation with $C^2$, i.e., for all $i=1,\dots, j_2$, $C_{2,i}=C_{v_i}$ for some $v_i \in \mathcal{I}$ with $|v_i|=|w^{m'}|$, see Figure \ref{fig1} for the case of a Cantor set in the plane. Let $S_{2,i},i=1,\dots,j_2,$ be the iterated similarities such that $C_{2,i}=S_{2,i}(C)$ and denote $C^1_{2,i}=S_{2,i}(C^1)$ and $C^2_{2,i}=S_{2,i}(C^2)$. Then exactly as before for $i=1,\dots,j_2,$ and $p \in C^2_{2,i}$,
$$\int_{C_{2,i}\stm C^1_{2,i}}K_\o (p,y)d \mu y= \int_{C\stm C^1}K_\o(S_{2,i}^{-1}(p),y)d \mu y >M$$
by (\ref{mest}) since $S_{2,i}^{-1}(p) \in C^2$.
\begin{figure}
\centering
\includegraphics[scale = 0.4]{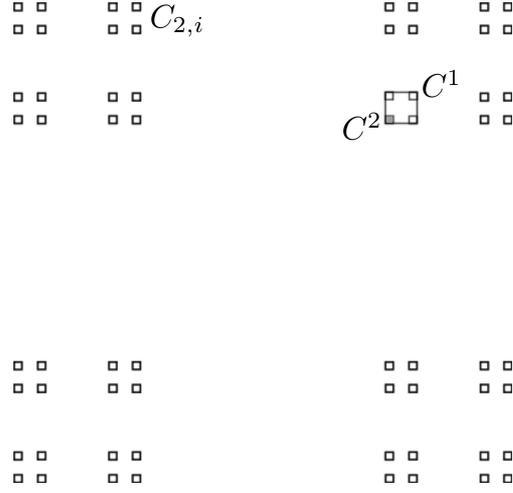}
\caption{The set $C$ is split into cylinders $C_{2,i}$ having the same length as the grey shaded cylinder $C^2$, which is contained in $C^1$.}
\label{fig1}
\end{figure}
Continuing the same splitting process, we can write for $n\geq 3$,
$$C \stm \left(C^2 \cup \bigcup_{i=1}^{j_2}C_{2,i}^2 \cup \dots \cup \bigcup_{i=1}^{j_{n-1}}C_{n-1,i}^2 \right)=\bigcup_{i=1}^{j_{n}}C_{n,i},$$
where for all $3 \leq k\leq n$:
\begin{enumerate}
\item The $C_{k,i}$'s for $i=1,\dots,j_k,$ are cylinder sets in the same generation with any $C_{k-1,i}^2,i=1,\dots,j_{k-1}$.
\item $C_{k,i}^1=S_{k,i} (C^1),i=1,\dots,j_k$ where by $S_{k,i}$ we denote the iterated map such that $S_{k,i}(C)=C_{k,i}$.
\item For all $p \in C_{k,i}^2=S_{k,i}(C^2)$,
\begin{equation}
\label{fmest}
\int_{C_{k,i}\stm C^1_{k,i}}K_\o(p,y)d \mu y>M.
\end{equation}
\end{enumerate}

Next we define the cylindrical maximal operator 
\begin{equation}
\label{cylmax}
T_C^\ast (f)(p)=\sup_{\substack{ v,w \in \mathcal{I} \\ p \in C_v \subset C_w}}\left|\int_{C_w \stm C_v}K_\o(p,y)f(y)d\mu y\right|
\end{equation}
for $p \in C$ and $f \in L^1 (\mu)$.
It follows by (\ref{fmest}) that for every $n \in \N,n \geq 2,$
\begin{equation}
\label{mtcyl}
T_C^\ast (1)(p)>M
\end{equation}
for $p \in C^2 \cup \bigcup_{k=2}^n \bigcup_{i=1}^{j_k} C^2_{k,i}.$

Let $\lambda= \frac{\mu(C^2)}{\mu(C)}<1$. Since $\mu(C \stm C^2)=(1-\lambda)\mu (C)$ it follows easily that for $n \in \N,n \geq 2,$
$$\mu \left(C \stm \left(C^2 \cup \bigcup_{k=2}^n \bigcup_{i=1}^{j_k}C_{k,i}^2\right) \right)=(1-\lambda)^n \mu(C).$$
If $n$ is chosen large enough such that $(1-\lambda)^n< \frac{1}{2}$,
$$\mu(\{p \in C:T_C^\ast (1)(p)>M\})\geq \mu(C^2 \cup \cup_{k=2}^n \cup_{i=1}^{j_k} C^2_{k,i})>\frac{1}{2} \mu (C). $$
This implies that 
$$\|T_C^\ast(1)\|_{L^\infty(\mu)}\geq M\ \text{and} \ \int (T_C^\ast (1))^2 d \mu > \frac{M^2 \mu(C)}{2}.$$
Since $M$ can be selected arbitrarily big, $\|T_C^\ast(1)\|_{L^\infty(\mu)}=\infty$ and the operator $T_C^\ast$ is unbounded in $L^2(\mu)$. 

Notice that there exists a constant $\alpha_C>0$, depending only on the set $C$, such that for every $v \in \mathcal{I}$ 
\begin{equation}
\label{disj}
\dist (C_v,C \stm C_v) \geq \alpha_C  \diam (C_v). 
\end{equation}
To see this first notice that since the sets $C_i= S_i(C),\ i=1,\dots,N,$ are disjoint there exists some $\alpha_C>0$ such that $\dist (C_i,C \stm C_i) \geq \alpha_C  \diam (C_i)$. Hence for all $v=(i_1,\dots,i_{|v|}) \in \mathcal{I}$, 
\begin{equation}
\begin{split}
\label{ssdisj}
\dist (C_v,C_{v|_{|v|-1}} \stm C_v)&=\dist (S_{v|_{|v|-1}}(S_{i_{|v|}}(C)),S_{v|_{|v|-1}}(C\stm S_{i_{|v|}}(C))) \\
&=r_{i_1} \dots r_{i_{|v|-1}} \dist(C \stm S_{i_{|v|}}(C))\\
& \geq \alpha_C r_{i_1} \dots r_{i_{|v|-1}} \diam(S_{i_{|v|}}(C))\\
&= \alpha_C r_{i_1} \dots r_{i_{|v|}} \diam(C)=\diam(C_v).
\end{split}
\end{equation}
Furthermore $C\stm C_v= \cup_{j=1}^{|v|} C_{v|_{j-1}} \stm C_{v|_j}$ and this union is disjoint. Therefore using (\ref{ssdisj})
\begin{equation*}
\begin{split}
\dist(C_v,C \stm C_v)&= \min_{j=1,\dots, |v|} \dist(C_v,C_{v|_{j-1}} \stm C_{v|_j}) \\
&\geq \min_{j=1,\dots, |v|} \dist(C_{v|_j},C_{v|_{j-1}} \stm C_{v|_j})\\
&\geq \alpha_C \min_{j=1,\dots, |v|} \diam(C_{v|_j}) \\
&=\alpha_C \diam (C_v).
\end{split}
\end{equation*}
\begin{lm}
\label{compop}
There is a constant $A_{\text{C}}$ depending only on the Cantor set $C$ and the kernel $K_\o$ such that for all $w,v \in \mathcal{I}$ and $p \in \hn$ for which $C_v \subset C_w$ and $\dist (p,C_v) \leq \frac{\alpha_C}{2}\diam (C_v)$, 
$$\left|\int_{C_w \stm C_v} K_\o(p,y)d \mu y \right|\leq \left| \int_{B(p,2\diam (C_w))\stm B(p,2\diam (C_v)) } K_\o(p,y)d \mu y \right|+ A_{C}.$$
\end{lm}
\begin{proof} 
We can always assume that $\alpha_C \leq 1$ and hence for $p \in \widetilde{C_v}:=\{p:\dist (p,C_v) \leq \frac{\alpha_C}{2}\diam (C_v)\}$
\begin{equation}
\label{longcyldif}
C_w \stm C_v=(C_w \stm B(p,2\diam (C_v))) \cup ((B(p,2\diam (C_v))\stm C_v)\cap C_w)
\end{equation}
and
\begin{equation} 
\begin{split}
\label{longballdif}
B(p,2\diam (C_w))\stm B(p,2\diam (C_v))&=(B(p,2\diam (C_w))\stm (C_w \\
&\quad \quad \cup B(p,2\diam (C_v))\cup (C_w \stm B(p,2\diam (C_v))).
\end{split}
\end{equation}
Using (\ref{longballdif}) we replace the term $C_w \stm B(p,2\diam (C_v))$ in (\ref{longcyldif}) and we estimate
\begin{equation*}
\begin{split}
\left|\int_{C_w \stm C_v} K_\o(p,y)d \mu y \right| &\leq \left| \int_{B(p,2\diam (C_w))\stm B(p,2\diam (C_v)) } K_\o(p,y)d \mu y \right|\\
& +\left| \int_{B(p,2\diam (C_w))\stm (C_w \cup B(p,2\diam (C_v))} K_\o(p,y)d \mu y \right|\\
& +\left| \int_{(B(p,2\diam (C_v))\stm C_v)\cap C_w} K_\o(p,y)d \mu y \right|.
\end{split}
\end{equation*}
Let,
$$I_1=\left| \int_{B(p,2\diam (C_w))\stm (C_w \cup B(p,2\diam (C_v))} K_\o(p,y)d \mu y \right| \ \text{and} \ I_2=\left| \int_{(B(p,2\diam (C_v))\stm C_v)\cap C_w} K_\o(p,y)d \mu y \right|. $$
Notice that $C_v \subset C_w$ implies that $\widetilde{C_v} \subset \widetilde{C_w}$. Hence for $p \in \widetilde{C_v}$ by (\ref{disj})
$$d(p,C\stm C_w)\geq d(C_w, C\stm C_w)-d(p,C_w) \geq \frac{\alpha_C}{2} d(C_w).$$
Therefore we can now estimate
\begin{equation*}
\begin{split}
I_1\leq \int_{B(p,2\diam (C_w))\stm C_w } \frac{{\| \o\|}_{L^\infty (\mu)}}{d(p,y)^s} d \mu y \leq \frac{{\|\o\|}_{L^\infty (\mu)} \mu (B(p,2\diam(C_w))}{2^{-s}{\alpha_C}^s \diam (C_w)^s}\leq \frac{4^{s}A_\mu {\|\o\|}_{L^\infty (\mu)}}{{\alpha_C}^s}.
\end{split}
\end{equation*}
Notice also that if $p \in  \widetilde{C_v}$ and $y \in B(p,2\diam (C_v))\stm C_v \cap \spt \mu$ then $d(p,y) \geq \frac{\alpha_C}{2}\diam(C_v)$. Hence in the same way
\begin{equation*}
\begin{split}
I_2\leq \int_{B(p,2\diam (C_v))\stm C_v}  \frac{{\|\o\|}_{L^\infty (\mu)}}{d(p,y)^s} d \mu y\leq \frac{4^{s}A_\mu {\|\o\|}_{L^\infty (\mu)}}{{\alpha_C}^s}.
\end{split}
\end{equation*}
\end{proof}
Lemma \ref{compop} implies that for all $p \in C$,
$$T_C^\ast (1) (p)\leq 2T_{K_\o}^\ast (1) (p)+A_{C},$$
therefore $\|T_{K_\o}^\ast (1)\|_{L^\infty (\mu)}=\infty$ and $T_{K_\o}^\ast$ is unbounded in $L^2(\mu)$.
\end{proof}

\begin{rem} Even when the ambient space is Euclidean, Theorem \ref{unb} provides new information about the behavior of general homogeneous singular integrals on self-similar sets. For example it follows easily that the operator associated to the kernel $z^3 /|z|^4, z \in \C \stm \{0\},$ is unbounded on many simple 1-dimensional self-similar sets, perhaps the most recognizable among them being the Sierpi\'nski gasket. Furthermore for any kernel $K_\o (x)=\frac{\o(x/|x|)}{|x|^s},x \in \Rn \stm \{0\}, s \in (0,n),$ where $\o$ is continuous one can easily find Sierpi\'nski-type $s$-dimensional self-similar sets $C_s$ for which one can check using Theorem \ref{unb} that the corresponding operator $T_{K_\o}$ is unbounded.
\end{rem}
\section{$\dh$-removability and singular integrals}
For an introduction to Heisenberg groups, see for example \cite{cap} or \cite{blu}. Below we state the basic facts needed in this paper.

The Heisenberg group $\hn$, identified with $\R^{2n+1}$, is a non-abelian group where the group
operation is given by,
$$p\cdot q=\left(p_1+q_1,\cdots,p_{2n}+q_{2n},p_{2n+1}+q_{2n+1}-2\sum_{i=1}^n(p_iq_{i+n}-p_{i+n}q_i)\right).$$
We will also denote points $p \in \hn$ by $p=(p',p_{2n+1}), p'\in\R^{2n}, p_{2n+1}\in\R$.
Recall that for any $q \in \hn$ and $r >0$, the left translations $\tau_q:\hn \ra \hn$ are given by 
$$\tau_q(p)=q\cdot p.$$
Furthermore we define the dilations $\delta_r:\hn \ra \hn$ by
$$\delta_{r}(p)=(rp_1,\dots,rp_{2n},r^2p_{2n+1}).$$
These dilations are group homomorphisms.

A natural metric $d$ on $\hn$ is defined by
$$d(p,q)=\|p^{-1}\cdot q\|$$
where
$$\|p\|=(|(p_1,\dots,p_{2n})|^4_{\R^{2n}}+p^2_{2n+1})^{\frac{1}{4}}.$$
The metric is left invariant, that is $d(q\cdot p_1,q\cdot p_2)=d(p_1,p_2)$, and the dilations satisfy 
$d(\delta_r(p_1),\delta_r(p_2))=rd(p_1,p_2)$. 

The $(2n+1)$-dimensional Lebesgue measure $\mathcal{L}^{2n+1}$ on $\hn$ is left and right invariant and it is a  Haar measure of the Heisenberg group. We stress that although the topological dimension of $\hn$ is $2n+1$ the Hausdorff dimension of $(\hn,d)$ is $Q:=2n+2$, see e.g. \cite{blu}, 13.1.4, which is also called the homogeneous dimension of $\hn$.

The Lie algebra of left invariant vector fields in $\hn$ is generated by
$$X_i:=\partial_i +2x_{i+n}\partial_{2n+1}, \  \ Y_i:=\partial_{i+n}-2x_{i}\partial_{2n+1}, \ \ T:=\partial_{2n+1},$$
for $i=1,\dots,n$. 

The vector fields $X_1,\dots,X_n,Y_1,\dots,Y_n$ define the horizontal subbundle $H\h^n$ of the tangent vector bundle of $\R^{2n+1}$. For every point $p \in \hn$ the horizontal fiber is denoted by $H \hn_p$ and can be endowed with the scalar product $\langle \cdot,\cdot \rangle_p$ and the corresponding norm $|\cdot|_p$ that make the vector fields $X_1,\dots,X_n,Y_1,\dots,Y_n$ orthonormal. Often when dealing with two sections $\varphi$ and $\psi$ whose argument is not stated explicitly we will use the notation $\langle \varphi, \psi \rangle$ instead of $\langle \varphi, \psi \rangle_p$. Therefore for $p,q \in \hn$, $\langle p,q \rangle= \langle p',q' \rangle_{\R^{2n}}$ and $|p|=|p'|_{2n}$. Furthermore for a given $p \in \hn$ we define the projections
$$\pi_p(q)=\sum_{i=1}^n q_i X_i(p)+\sum_{i=1}^n q_{i+n} Y_i(p)\ \text{for} \ q\in \hn.$$

\begin{df}
\label{subunit}
An absolutely continuous curve $\gamma:[0,T]\ra \hn$ will be called sub-unit, with respect to the vector fields $X_1,\dots,X_n,Y_1,\dots,Y_n,$ if there exist real measurable functions $a_1(t),\dots,a_{2n}(t),t \in [0,T]$, such that $\sum_{j=1}^{2n}  a_j(t)^2 \leq 1$ and 
$$\dot \g (t)=\sum_{j=1}^n a_j(t)X_j(\g (t))+\sum_{j=1}^n a_{j+n}(t)Y_j(\g (t)),\ \ \text{for a.e.}\ t\in[0,T].$$
\end{df}

\begin{df}
\label{cc}
For $p,q \in \hn$ their Carnot-Carath\'eodory distance is
\begin{equation*}
\begin{split}
d_c(p,q)=\inf \{T>0:\ &\text{there is a subunit curve} \ \g:[0,T]\ra \hn \\ 
&\text{such that} \ \g(0)=p \ \text{and} \ \g(T)=q\}.
\end{split}
\end{equation*}
\end{df}

\begin{rem} It follows by Chow's theorem that the above set of curves joining $p$ and $q$ is not empty and hence $d_c$ is a metric on $\hn$. Furthermore the infimum in the definition can be replaced by a minimum. See \cite{blu} for more details.
\end{rem}

As well as with $d$ the metric $d_c$ is left invariant 
and homogeneous with respect to dilations,
see, for example, Propositions 5.2.4 and 5.2.6 of \cite{blu}. The closed and open balls with respect to $d_c$ will be denoted by $B_c(p,r)$ and $U_c(p,r)$. 

The following result is well known and can be found for example in \cite{blu} and \cite{cap}. 

\begin{pr}
\label{equiv}
The Carnot-Carath\'eodory distance $d_c$ is globally equivalent to the metric $d$.
\end{pr}
If $f$ is a real function defined on an open set of $\hn$ its $\h$-gradient is given by
$$\nabla_\h f=(X_1f,\dots,X_n f,Y_1 f,\dots,Y_n f).$$
The $\h$-divergence of a function $\phi=(\phi_1,\dots,\phi_{2n}):\hn \ra \R^{2n}$ is defined as
$$\di_\h \phi=\sum_{i=1}^n(X_i\phi_i+Y_i \phi_{i+n}).$$
The sub-Laplacian in $\hn$ is given by
$$\Delta_\h= \sum_{i=1}^n (X_i^2+Y_i^2)$$
or equivalently 
$$\D_\h=\divh \nabla_\h .$$ 

\begin{df}
\label{harmdef}Let $D\subset \hn$ be an open set. A real valued function $f \in C^2(D)$ is called $\dh$-harmonic, or simply harmonic, on $D$ if $\dh f= 0$ on $D$.
\end{df}

Actually, the assumption $f \in C^2(D)$ is superfluous, since even the distributional solutions of $\dh f= 0$ are $C^{\infty}$, see \cite{blu}. 

\begin{df}
\label{pansud}
Let $D$ be an open subset of $\hn$. We say that $f:D \ra \R$ is Pansu differentiable at $p \in D$ if there exists a homomorphism $L:\hn \ra \R$ such that
$$\lim_{r \ra 0^+}\frac{f(\t_p (\delta_r \nu))-f(p)}{r}=L(\nu)$$ 
uniformly with respect to $\nu$ belonging to some compact subset of $\hn$. Furthermore, $L$ is unique and we write $L:=d_\h f(p)$.
\end{df}
The proof of the following proposition, as well as a comprehensive discussion about calculus in $\hn$, can be found in \cite{fss}.
\begin{pr}
\label{pansupro}
If $f$ is Pansu differentiable at $p$, then 
$$d_\h f(p)(\nu)=\langle \anah f (p), \pi_p(\nu)\rangle_p.$$
\end{pr}

We shall consider removable sets for Lipschitz solutions of the sub-Laplacian:

\begin{df}
\label{rem} A compact set $C \subset \hn$ will be called removable, or $\dh$-removable for Lipschitz $\dh$-harmonic functions, if for every domain $D$ with $C \subset D$ and every Lipschitz function $f:D \ra \R$,
$$\dh f=0 \ \text{in} \  D\stm C \ \text{implies} \ \dh f = 0 \ \text{in} \ D.$$
\end{df}
As usual we denote for any $D \subset \hn$ and any function $f:D \ra \R$,
$$\text{Lip} (f):= \sup_{x,y \in D} \frac{|f(x)-f(y)|}{d(x,y)},$$
and we will also use the following notation for the upper bound for the Lipschitz constants in  Carnot-Carath\'eodory balls:
$$\text{Lip}_{\text{B}} (f):= \sup \{\text{Lip} (f|_{U_c(p,r)}):p \in D, r>0, U_c(p,r) \subset D\}.$$

The following proposition is known. It follows, for example, from the Poincar\'e inequality, see Theorem 5.16 in \cite{cap} and the arguments for its proof on pages 106-107. It is also essentially contained in a more general setting in Theorems 1.3 and 1.4 of \cite{gn}. However, we prefer to give a simple direct proof.

\begin{pr}
\label{panulip}
Let $D \subset \hn$ be a domain and let $f \in C^1(D)$. Then $\text{Lip}_{\text{B}} (f) < \infty$ if and only if $\|\anah f\|_\infty < \infty$. More precisely, there is a constant $c(n)$ depending only on $n$ such that
\begin{equation}\label{panulip1}
\|\anah f\|_\infty\leq \text{Lip}_{\text{B}} (f)\leq c(n)\|\anah f\|_\infty.
\end{equation}
\end{pr}
\begin{proof}
By Pansu's Rademacher type theorem, see \cite{pan}, $f$ is a.e. Pansu differentiable in $D$. Let $q$ be a point where $f$ is Pansu differentiable, then for all $\nu \in \hn$,
$$\left|d_\h f(q)(\nu)\right|=\lim_{r \ra 0^+}\left|\frac{f(q\cdot \delta_r(\nu))-f(q)}{r}\right| \leq \text{Lip}_{\text{B}} (f) \|\nu\|.$$ 
By Proposition \ref{pansupro},
$$d_\h f(q)(\nu)= \langle \anah f(q), \pi_q(\nu)\rangle_q= \langle \anah f(q),\nu'\rangle_{\R^{2n}},$$
and choosing $\nu=(\anah f(q),0)$ we get $|\anah f (q)|\leq \text{Lip}_{\text{B}} (f)$, and so $\|\anah f\|_\infty\leq \text{Lip}_{\text{B}} (f)$. 

On the other hand we check that if $\|\anah f\|_\infty < \infty$, then $\text{Lip}_{\text{B}} (f) < \infty$. For any $q \in D$ there exists a radius $r_q$ such that $U_c(q,r_q) \subset D$. Then by the definition of the Carnot-Carath\'eodory metric for any $p \in U_c(q,r_q)$ there exists a subunit curve  $\gamma:[0,T]\ra U_c(q,r_q)$, as in Definition \ref{subunit}, such that $\gamma(0)=q,\gamma(T)=p$ and $T=d_c(q,p)$. Then,
\begin{equation*}
\begin{split}
|f(q)-f(p)|&=\left|\int_0^T\frac{d}{dt}(f(\gamma(t)))dt\right| =\left|\int_0^T\langle\nabla f (\gamma(t)),\dot{\gamma}(t)\rangle dt \right| \\
&\leq \int_0^T\left| \sum_{j=1}^n a_j(t) \langle \nabla f(\gamma(t)),X_j (\gamma(t))\rangle+a_{j+n}(t) \langle\nabla f(\gamma(t)),Y_j (\gamma(t))\rangle \right|dt \\
&\leq \int_0^T \left( \sum_{j=1}^{2n}a_j(t)^2\right)^{\frac{1}{2}} \left(\sum_{j=1}^n \langle\nabla f (\g(t)),X_j (\gamma (t)) \rangle ^2+ \langle\nabla f (\g(t)),Y_j (\gamma (t)) \rangle^2 \right)^{\frac{1}{2}}dt\\
&\leq \int_0^T ( \sum_{j=1}^n (X_j f(\gamma(t)))^2 + (Y_j f(\gamma(t)))^2)^{\frac{1}{2}}dt \\
&= \int_0^T |\anah f (\gamma (t))|dt \\
&\leq \|\anah f\|_\infty d_c(q,p),
\end{split}
\end{equation*}
where in the fourth line we used that 
$$\langle\nabla f (\g(t)),X_j (\gamma (t)) \rangle=X_j (f)(\g(t)) \ \text{and} \ \langle\nabla f (\g(t)),Y_j (\gamma (t)) \rangle=Y_j(f )(\g(t)).$$
The inequality  $\text{Lip}_{\text{B}} (f)\leq c(n)\|\anah f\|_\infty$ follows from this and Proposition \ref{equiv}. 
\end{proof}

Fundamental solutions for sub-Laplacians in homogeneous Carnot groups are defined in accordance with the classical Euclidean setting. In particular in the case of the sub-Laplacian in $\hn$:
\begin{df}[Fundamental solutions] A function $\Gamma : \R^{2n+1} \stm \{0\} \ra \R$ is a fundametal solution for $\dh$ if:
\begin{enumerate}
\item $\G \in C^{\infty}(\R^{2n+1} \stm \{0\})$,
\item $\G \in L_{\text{loc}}^1(\R^{2n+1})$ and $\lim_{\|p\| \ra \infty}\G(p)\ra 0$,
\item for all $\varphi \in C_0^{\infty} (\R^{2n+1})$,
$$\int_{\R^{2n+1}}\G(p) \dh \varphi (p) dp=-\varphi (0).$$
\end{enumerate}
\end{df}

It also follows easily, see Theorem 5.3.3 and Proposition 5.3.11 of \cite{blu}, that for every $p \in \hn$, 
\begin{equation}
\label{fundconv}
\G \ast \dh \varphi (p)=-\varphi(p) \ \text{for all} \ \varphi \in C_0^{\infty} (\R^{2n+1}).
\end{equation}
Convolutions are defined as usual by $$f\ast g (p)=\int f(q^{-1}\cdot p)g(q)dq$$
for $f,g \in L^1$ and $p \in \hn$.

One very general result due to Folland, see \cite{fol}, guarantees that there exists a fundamental solution for all sub-Laplacians in homogeneous Carnot groups with homogeneous dimension $Q>2$. In particular the fundamental solution $\G$ of $\dh$ is given by
$$\G(p)=C_Q \|p\|^{2-Q}\ \text{for} \ p \in \hn \stm \{0\}$$
where $Q=2n+2$ is the homogeneous dimension of $\hn$. The exact value of $C_Q$ can be found in  \cite{blu}.

Let $K=\anah \Gamma$, then $K=(K_1,\dots,K_{2n}):\hn \ra \R^{2n}$ where 
\begin{equation}
\label{lapkernels}
\begin{split}
K_i(p)=c_Q \frac{p_i|p'|^2+p_{i+n}p_{2n+1}}{\|p\|^{Q+2}} \ \text{and} \
K_{i+n}(p)=c_Q \frac{p_{i+n}|p'|^2-p_{i}p_{2n+1}}{\|p\|^{Q+2}},
\end{split}
\end{equation}
for  $i=1,\dots,n$, $p \in \hn \stm \{0\}$ and $c_Q=(2-Q)C_Q$.  We will also use the following notation,
\begin{equation}
\label{omega}
\begin{split}
\o_i(p)=c_Q \frac{p_i|p'|^2+p_{i+n}p_{2n+1}}{\|p\|^3} \ \text{and} \
\o_{i+n}(p)=c_Q \frac{p_{i+n}|p'|^2-p_{i}p_{2n+1}}{\|p\|^3},
\end{split}
\end{equation}
for  $i=1,\dots,n$ and $p \in \hn \stm \{0\}$.
Hence,
\begin{equation}
\label{vecker}
\begin{split}
K_i(p)=\frac{\o_i(p)}{\|p\|^{Q-1}} \ \text{and} \ K(p)=\frac{\o(p)}{\|p\|^{Q-1}},
\end{split}
\end{equation}
for  $i=1,\dots,2n,\o=(\o_1,\dots,\o_{2n})$ and $p \in \hn \stm \{0\}$. It follows that the functions $\o_i$ are homogeneous and hence, recalling Definition \ref{hom}, the kernels $K_i$ are $(Q-1)$-homogeneous.

The following proposition asserts that $K$ is a standard kernel.
\begin{pr}
\label{stanker}
For all $i=1,\dots,2n$, 
\begin{enumerate}
\item $|K_i (p)| \lesssim \|p\|^{1-Q}$ for $p \in \hn \stm \{0\}$,
\item $|\anah K_i (p)| \lesssim  \|p\|^{-Q}$ for $p \in \hn \stm \{0\}$,
\item $|K_i (p^{-1} \cdot q_1)-K_i (p^{-1} \cdot q_2)| \lesssim  \max \left\{\dfrac{d(q_1,q_2)}{d(p,q_1)^Q},\dfrac{d(q_1,q_2)}{d(p,q_2)^Q} \right\}$ for $q_1,q_2 \neq p \in \hn$.
\end{enumerate}
\end{pr}
\begin{proof}
The size estimate (i) follows immediately by the definition of the kernel $K$. 
It also follows easily that for $p \in \hn \stm \{0\}$,
$$|\partial_j K_i (p)| \lesssim \frac{1}{\|p\|^{Q}} \ \text{for} \ j,i=1,\dots,2n ,$$
and
$$|\partial_{2n+1} K_i (p)| \lesssim \frac{1}{\|p\|^{Q+1}} \ \text{for} \ i=1,\dots,2n.$$
Hence 
$$|X_i K_j (p)| \lesssim \frac{1}{\|p\|^{Q}} \ \text{and} \ |Y_i K_j (p)|\lesssim \frac{1}{\|p\|^{Q}}\ \text{for} \ i=1,\dots,n,\ j=1,\dots,2n,$$
and (ii) follows.

For the proof of (iii) let $q_1,q_2 \neq p \in \hn$. Without loss of generality assume that $d_c(q_1,p)\leq d_c(q_2,p)$. We are going to consider two cases.

\textit{First Case:} $d_c(q_1,q_2)\geq \frac{1}{2} d_c (q_1,p)$ \\
Since $d_c$ is globally equivalent to $d$ we can use (i) to obtain
\begin{equation*}
\begin{split}
|K_i (p^{-1} \cdot q_1)-K_i (p^{-1} \cdot q_2)| & \lesssim \frac{1}{d_c(q_1,p)^{Q-1}}+\frac{1}{d_c(q_2,p)^{Q-1}}\\
&\leq \frac{ 2}{d_c(q_1,p)^{Q-1}} \\
&\leq \frac{4 d_c(q_1,q_2)}{d_c(q_1,p)^{Q}}\\
&\lesssim \max \left\{\dfrac{d(q_1,q_2)}{d(p,q_1)^Q},\dfrac{d(q_1,q_2)}{d(p,q_2)^Q} \right\}.
\end{split}
\end{equation*}

\textit{Second Case}: $d_c(q_1,q_2)< \frac{1}{2} d_c (q_1,p)$ \\
By the definition of the Carnot-Carath\'eodory metric there is a sub-unit curve $\g:[0,d_c(q_1,q_2)]\ra \hn$ such that $\g (0)=p^{-1}\cdot q_1$ and $\g (d_c(q_1,q_2))=p^{-1} \cdot q_2$. Furthermore,
\begin{equation}
\label{refe}
\g([0,d_c(q_1,q_2)]) \subset B_c(p^{-1}\cdot q_1, d_c (q_1,q_2)).
\end{equation}
Hence for every $t \in [0,d_c(q_1,q_2)]$, since by (\ref{refe}) $d_c(\g(t),p^{-1}\cdot q_1) \leq d_c(q_1,q_2)$,
\begin{equation}
\label{gammacomp}
\begin{split}
\|\g(t)\|\gtrsim d_c(0,\g(t)) &\geq d_c (0, p^{-1}\cdot q_1)-d_c(\g(t),p^{-1}\cdot q_1) \\
&\geq d_c(p,q_1)-d_c(q_1,q_2)\\
&\geq \frac{1}{2} d_c(q_1,p).
\end{split}
\end{equation}
Therefore if $T=d_c(q_1,q_2)$  we can estimate as in Proposition \ref{panulip} for $i=1,\dots,2n$:
\begin{equation*}
\begin{split}
|K_i (p^{-1} \cdot q_1)-K_i (p^{-1} \cdot q_2)|&=|K_i (\g (0))-K_i (\g (T))| \\
&=\left| \int_0^T  \frac{d}{dt} (K_i (\g(t)) dt \right| \\
&\leq \int_0^T\left(\sum_{j=1}^n (X_j (K_i)(\g(t)))^2+(Y_j(K_i )(\g(t)))^2 \right)^\frac{1}{2} dt\\
&=\int_0^T|\anah K_i (\g(t))|dt\\
&\lesssim \int_0^T \frac{dt}{\|\g(t)\|^Q} \\
&\lesssim \frac{d_c(q_1,q_2)}{d_c(q_1,p)^Q}\\
&\lesssim \max \left\{\dfrac{d(q_1,q_2)}{d(p,q_1)^Q},\dfrac{d(q_1,q_2)}{d(p,q_2)^Q} \right\}.
\end{split}
\end{equation*}
where we used (ii) and (\ref{gammacomp}) respectively.
\end{proof}
In the following we prove a representation theorem for Lipschitz harmonic functions of $\hn$ outside a compact set of finite $\ha^{Q-1}$ measure.

\begin{thm}
\label{main}
Let $C$ be a compact subset of $\hn$ with  $\mathcal{H} ^{Q-1} (C) < \infty$ and let $D\supset C$ be a domain in $\hn$. Suppose $f:D \ra \R$ is a Lipschitz function such that  $\dh f = 0$ in $D \stm C$.
Then there exist a bounded domain $G, C \subset G \subset D$, a Borel function $h:C\ra \R$ and a $\dh$-harmonic function $H:G \ra \R$ such that 
$$f(p)=\int_C \G(q^{-1}\cdot p)h(q) d \ha^{Q-1}q+H(p) \ \text{for} \ p \in G \stm C$$
and $\|h\|_{L^\infty(\ha^{Q-1} \lfloor C)} +\|\anah H \|_\infty  \lesssim 1$.
\end{thm}
\begin{proof}
It suffices to prove the theorem with $\ha^{Q-1}$ replaced by $\ss^{Q-1}$. Without loss of generality we can assume that $D$ is bounded. Let $D_1$ be some domain such that $C \subset D_1 \subset D$ and $\dist (\overline{D_1},\hn \stm D)>0$. For every $m=1,2,\dots$ there exists a finite number $j_m$ of balls $U_{m,j}:=U(p_{m,j},r_{m,j})$ such that $U_{m,j}\cap C \neq \emptyset$,
\begin{equation}
\label{gmmeas}
C \subset \bigcup_{j=1}^{j_m} U_{m,j} \subset D_1, \ \ r_{m,j} \leq \frac{1}{m}\ \text{and} \ \sum_{j=1}^{j_m} r_{m,j}^{Q-1} \leq \ss ^{Q-1} (C)+ \frac{1}{m}.
\end{equation}
Furthermore let $G_m = \cup_{j=1}^{j_m} U_{m,j}$ and $$0<\e_m <\min\{\dist(C,\hn \stm G_m),\dist (\overline{D_1},\hn \stm D) \}.$$
By the Whitney-McShane extension Lemma there exists a Lipschitz function $F:\hn \ra \R$ such that $F|_{D}=f$ and $F$ is bounded.

Let $J \in C^\infty_0 (\R^{2n+1}),J \geq 0$, such that $\spt J \subset B(0,1)$ and $\int J =1$. For any $\e>0$ let $J_\e(x)=\e^{-Q}J (\delta_{1/\e}(x))$. We consider the following sequence of mollifiers,
\begin{equation}
\label{mol}
\begin{split}
f_m(x):&=F \ast J_{\e_m}(x)=\int F(y) J_{\e_m}(x \cdot y^{-1})dy \\
&=\int F(y^{-1}\cdot x) J_{\e_m}(y)dy,
\end{split}
\end{equation}
for $x \in \hn$. Since $F$ is bounded and uniformly continuous
$$\|f_m-F\|_\infty \ra 0$$ 
and furthermore for all $m \in \N$, 
\begin{enumerate}
\item $f_m \in C^\infty$,
\item $\|\anah f_m\|_\infty \leq \|\anah F\|_\infty < \infty$,
\item $f_m$ is harmonic in the set $$D_{\e_m}:=\{p \in D \stm C : \dist(p, \partial (D \stm C))> \e_m\}.$$
\end{enumerate}
It follows from (iii) that every mollifier $f_m$ is harmonic in the set $D_1 \stm G_m$.
We continue by choosing another domain $D_2$ such that $G_m \subset D_2 \subset D_1$ for all $m=1,2,\dots$, and an auxiliary function $\varphi \in C_0^\infty(\R^{2n+1})$ such that
$$ \varphi  = \begin{cases}
       1 & \text{in} \ D_2\\
       0 & \text{in} \ \hn \stm \overline{D}_1.\\
     \end{cases}$$
For $m=1,2,\dots$ set $g_m:= \varphi f_m$ and notice that $g_m \in C_0^\infty(\R^{2n+1})$ and 
$$\|\anah g_m\|_\infty \leq A_1$$
where $A_1$ does not depend on $m$. It follows by (\ref{fundconv}) that for all $m \in \N$,
\begin{equation}
\label{convgm}
-g_m(p)=\G \ast \dh g_m (p)\ \text{for all} \  p \in \hn.
\end{equation}
Notice that
\begin{enumerate}
\item $g_m=0$ in $\hn \stm \overline{D_1}$,
\item $g_m=f_m$ in $D_2 \stm G_m$ and hence $\dh g_m=\dh f_m=0$ in $D_2 \stm G_m$.
\end{enumerate}
Therefore for all $m \in \N$ and $p \in D_2 \stm G_m$,
\begin{equation}
\label{finconv}
-f_m (p)=\int_{G_m} \G (q^{-1} \cdot p) \dh g_m (q) dq+\int_{\overline{D_1} \stm D_2} \G(q^{-1} \cdot p)  \dh g_m (q) dq
\end{equation}
by (\ref{convgm}).

For $m \in \N$ set $H_m: D_2 \ra \R$,
\begin{equation}
\label{hfun}
H_m(p)= -\int_{\overline{D_1} \stm D_2} \G(q^{-1} \cdot p)  \dh g_m (q) dq
\end{equation}
and $I_m:D_2 \stm G_m \ra \R,\ m=1,2,\dots$,
\begin{equation}
\label{imfun}
I_m (p)= -\int_{G_m} \G (q^{-1} \cdot p) \dh g_m (q) dq.
\end{equation}
Since the functions $\dh g_m$ are uniformly bounded in $\overline{D_1} \stm D_2$, for all $m \in \N$  
\begin{enumerate}
\item $H_m$ is harmonic in $D_2$,
\item $\|\anah H_m\|_\infty \lesssim 1$, since $\anah \G$ is locally integrable.
\end{enumerate}
The functions $H_m$ are $C^{\infty}$ by H\"ormander's theorem, see for example Theorem 1 in Preface of \cite{blu}. Thus we can apply 
Proposition \ref{panulip} and conclude from (ii) that $\text{Lip}_{\text{B}} (H_m) \lesssim 1$. 

The functions $I_m$ can be expressed as,
\begin{equation}
\label{im1}
I_m (p)=-\int_{G_m} \di_{\h,q} (\G (q^{-1} \cdot p) \anah g_m (q))dq+\int_{G_m}\langle \anah \G (p^{-1} \cdot q), \anah g_m (q) \rangle dq,
\end{equation}
where $\di_{\h,q}$ stands for the $\h$-divergence with respect to the variable $q$ and we also used the left invariance of $\anah$ and the symmetry of $\G$ to get that $$\nabla_{\h,q}(\G(q^{-1} \cdot p))=\nabla_{\h,q}(\G(p^{-1} \cdot q))=\anah \G (p^{-1} \cdot q).$$
By the Divergence Theorem of Franchi, Serapioni and Serra Cassano, see \cite{fss} (in particular Corollary 7.7 ), 
\begin{equation}
\label{divfss}
\begin{split}
-\int_{G_m} \di_{\h,q} (\G (q^{-1} \cdot p) &\anah g_m (q)) dq \\
&=A_2 \int _{\partial G_m} \G (q^{-1} \cdot p) \langle \anah g_m(q), \nu_m (q) \rangle b(q)d \ss^{Q-1}q
\end{split}
\end{equation}
where $\nu_m$ is an $\ss^{Q-1}$-measurable section of $H \hn$ such that $|\nu_m(q)|=1$ for $\ss^{Q-1}$-a.e $q \in G_m$ and $b$ is a non-negative Borel function such that $\|b\|_{L^\infty(\ss^{Q-1})} \leq A_3$.
\footnote{The divergence theorem in \cite{fss} is stated in terms of the spherical Hausdorff measure $\ss^{Q-1}_\infty$ with respect to the norm $\|p\|_\infty:= \max \{ |p'|, \sqrt{|p_{2n+1}|}\}$. Since the corresponding norm $d_\infty$ is globally equivallent to $d$ we get that $ \ss^{Q-1}<<\ss^{Q-1}_\infty << \ss^{Q-1}$ and the function $b$ is  the Radon-Nikodym derivative $\frac{d \ss^{Q-1}_\infty}{d \ss^{Q-1}}$.}

By (\ref{gmmeas}), $\mathcal{L}^{2n+1}(G_m)\ra 0$, therefore for $p \in D_2 \stm C$,
\begin{equation}
\label{zerolim}
\lim_{m \ra \infty}\left| \int_{G_m}\langle \anah \G (p^{-1} \cdot q), \anah g_m (q) \rangle dq \right| \ra 0,
\end{equation}
since $|\anah g_m|$ is uniformly bounded in $D_2$ and $\anah \G$ is locally integrable.

Notice that the signed measures,
\begin{equation}
\label{sigmes}
\sigma_m=A_2 \langle \anah g_m(\cdot), \nu_m (\cdot) \rangle b  \ss^{Q-1} \lfloor \partial G_m,
\end{equation}
have uniformly bounded total variations $\|\sigma_m\|.$
This follows by (\ref{gmmeas}), as
\begin{equation}
\label{smvar}
\begin{split}
\|\sigma_m \| &\leq A_2 \| \anah g_m\|_\infty \| b\|_{L^\infty(\ss^{Q-1})}\ss^{Q-1}(\partial G_m) \\ 
&\leq A_1  A_2 A_3 \sum_j \ss^{Q-1} (\partial U_{m,j}) \\
&= A_4 \sum_j \alpha (Q-1) r_{m,j}^{Q-1} \\
&\leq A_5( \ss^{Q-1}(C) + \frac{1}{m}),
\end{split}
\end{equation}
for $A_4 :=A_1A_2A_3$, $A_5=\a (Q-1)A_4$ and $\a(Q-1):= \ss^{Q-1} (\partial B(0,1)).$  Therefore, by a general compactness theorem, see e.g. \cite{afp}, we can extract a weakly converging subsequence $(\sigma_{m_k})_{k \in \N}$ such that
$$\sigma_{m_k}\ra \sigma.$$
Furthermore $\spt \sigma:=\spt|\sigma| \subset C$. To see this let $p \notin C$. Let $\delta=\dist (p,C)$ and choose $i_0$ big enough such that $1 /m_{i_0}<\delta /4 $. Then by (\ref{gmmeas}), $p \notin \partial G_{m_i}$ for all $i \geq i_0$. Since $\spt \sigma_{m_i} \subset \partial G_{m_i}$ and $B(p, \frac{\delta}{2}) \cap \overline{G_{m_i}}=\emptyset$, 
$$|\sigma|(U(p,\delta / 2)) \leq \liminf_{i \ra \infty} |\sigma_{m_i}| (U(p,\delta / 2))=0,$$ 
which implies that $p \notin \spt \sigma$.

Notice also that by (\ref{smvar})
\begin{equation}
\label{sigvar}
\|\sigma\| \leq \liminf_{k \ra \infty} \|\sigma_{m_k}\| \leq A_5 \ss^{Q-1} (C).
\end{equation}

Finally combining (\ref{im1})-(\ref{sigmes}) we get that for $p \in D_2 \stm C$,
$$\lim_{k \ra \infty} I_{m_k} (p)= \int_C \G(q^{-1} \cdot p) d \sigma q$$
and by (\ref{finconv})-(\ref{imfun})
$$f(p)=\int_C \G(q^{-1} \cdot p) d \sigma q+\lim_{k \ra \infty} H_{m_k}(p).$$
Since the sequence of harmonic functions $(H_{m_k})$ is equicontinuous on compact subsets of $D_2$, the Arzela-Ascoli theorem implies that there exists a subsequence $(H_{m_{k_l}})$ which converges uniformly on compact subsets of $D_2$. From the Mean Value Theorem for sub-Laplacians and its converse, see \cite{blu}, Theorems 5.5.4 and 5.6.3, we deduce that $(H_{m_{k_l}})$ converges to a function $H$ which is harmonic in $D_2$. Therefore for $p \in D_2 \stm C$,
$$f(p)=\int_C \G(q^{-1} \cdot p) d \sigma q+H (p).$$
Furthermore the function $H$ is $C^{\infty}$ in $D_2$ with $\text{Lip}_{\text{B}}(H) \lesssim 1$, therefore by Proposition \ref{panulip} 
$$\|\anah H\|_\infty \lesssim 1.$$

Set $\mu= \ss^{Q-1} \lfloor C$. In order to complete the proof it suffices to show that 
\begin{equation}
\label{abscont}
\sigma  \ll \mu \ \text{and} \ h:=\frac{d \sigma}{d \mu} \in L^{\infty} (\mu).
\end{equation}
The proof of (\ref{abscont}) is almost identical with the proof appearing in \cite{mpa} but we provide the details for completeness.
It is enough to prove that for every open ball $U$ and its closure $\overline{U}$
\begin{equation}
\label{difsm}
|\sigma| (U) \leq A_5 \mu (\overline{U}).
\end{equation}
Then from (\ref{difsm}) we deduce that for any closed ball $B$ and open balls $U_i \supset B, U_i \ra B$,
\begin{equation}
|\sigma| (B) \leq \lim_{i \ra \infty} |\sigma|(U_i) \leq \lim_{i \ra \infty}  A_5 \mu (\overline{U_i})=A_5 \mu (B),
\end{equation}
which implies (\ref{abscont}).

Suppose that there exist an open ball $U$ and a positive number $\e$ such that
\begin{equation}
\label{contr}
|\s (U)|  >A_5 (\mu (\overline{U})+\e).
\end{equation}
In case $C \subset \cu$, (\ref{sigvar}) implies that $|\sigma| (U) \leq A_5 \mu (\overline{U})$ therefore we can assume that $C\stm \cu \neq \emptyset$. There exists a compact set $F$ such that 
\begin{equation}
 \label{ac12}
 F \subset C \stm \cu \ \text{and} \ \mu(F) >\mu(C\stm \cu)-\frac{\e}{4}.
\end{equation}
Let $\de_\e:=\dist(F, \cu)$ and choose $k \in \N$ large enough such that 
$1/ m_k  <\min \{\de_\e / 4,\e / 2 \}$. Then by (\ref{gmmeas})
\begin{equation}
\label{ac34}
\max_{j\leq j_{m_k}} r_{m_k,j} \leq \frac{1}{m_k} < \frac{\de_\e}{4}\ \text{and} \ \sum_{j=1}^{j_{m_k}} r_{m_k,j}^{Q-1} \leq \mu (C)+ \frac{1}{m_k}.
\end{equation}
Let
$$J_k^1=\{j: U_{m_k,j} \cap \cu \neq \emptyset \}, \ J_k^2=\{j: U_{m_k,j} \cap F \neq \emptyset \}.$$
It follows that $F \subset \cup_{j \in J_k^2} U_{m_{k,j}}$, therefore $\sum_{j \in J_k^2}  r_{m_{k,j}}^{Q-1} \geq \ss_{1/m_k}^{Q-1} (F)$. Choosing $k$ large enough
\begin{equation}
\label{ac5}
\sum_{j \in J_k^2}  r_{m_{k,j}}^{Q-1} \geq \mu (F)-\frac{\e}{4}.
\end{equation}
It also holds that
$$\overline{U}_{m_k,j_1} \cap \overline{U}_{m_k,j_2}= \emptyset \ \text{for} \ j_1 \in J_{k}^1,j_2 \in J_{k}^2.$$
Therefore for $k$ large enough, by (\ref{ac34}),
$$ \sum_{j \in J_k^1}  r_{m_{k,j}}^{Q-1}+\sum_{j \in J_k^2}  r_{m_{k,j}}^{Q-1} \leq \sum_{j=1}^{j_{m_k}}  r_{m_{k,j}}^{Q-1} \leq \mu (C) + \frac{\e}{2},$$
and by (\ref{ac5}) and (\ref{ac12})
\begin{equation}
\begin{split}
\label{325}
\sum_{j \in J_k^1}  r_{m_{k,j}}^{Q-1} &\leq \mu(C)-\mu(F)+\frac{3 \e}{4} \\
&<\mu(C)-\mu(C \stm \cu)+\e\\
&=\mu (\cu)+\e.
\end{split}
\end{equation}
For all $k \in \N$ large enough by the definition of $\s_{m_k}$, (\ref{sigmes}), and (\ref{325}) we see as in (\ref{smvar}) that 

\begin{equation*}
\begin{split}
|\s_{m_k}|( U) &\leq |\s_{m_k}|(\bigcup_{j \in J_k^1}\overline{U}_{m_{k,j}}) \\
&\leq A_4 \sum_{j \in J_k^1} \ss^{Q-1}( \partial U_{m_{k,j}}) \\
&= A_4 \sum_{j \in J_k^1} \a (Q-1) r_{m_k,j}^{Q-1} \\
&\leq A_5 (\mu (\cu)+\e).
\end{split}
\end{equation*}
Since $\s_{m_k} \ra \s$, we deduce that
$$|\s|(U) \leq \liminf_{k \ra \infty} |\s_{m_k}|(U)\leq A_5(\mu (\overline{U})+\e)$$
which contradits (\ref{contr}) and thus the proof is complete.
\end{proof}

The following theorem, with $Q$ replaced by $n$, is also valid for Lipschitz harmonic functions in $\Rn$.
\begin{thm}
\label{pos}
Let $C$  be a compact subset of $\hn$.
\begin{enumerate} 
\item If $\mathcal{H}^{Q-1}(C)=0$, $C$ is removable.
\item If $\dim C>Q-1$, $C$ is not removable.
\end{enumerate}
\end{thm}
\begin{proof} The first statement follows from Theorem \ref{main}. To see this let $D \supset C$ be a subdomain of $\hn$. Applying the previous Theorem we deduce that if $f:D \ra \R$ is Lipschitz in $D$ and $\dh$-harmonic in $D \stm C$ there exists a $\dh$-harmonic function $H$ in a domain $G, C\subset G \subset D$ such that
$$f(p)=-H(p) \ \text{for} \ p\in G\stm C.$$
This implies that $f=H$ in $G$. Hence $f$ is harmonic in $G$ (and so also in $D$). Therefore $C$ is removable.

In order to prove (ii) let $Q-1<s<\dim C$. By Frostman's lemma in compact metric spaces, see \cite{M}, there exists a Borel measure $\mu$ with $\spt \mu \subset C$ such that
$$\mu (B(p,r))\leq r^s \ \text{for} \ p \in \hn, r>0.$$
We define $f:\hn \ra \R^+$ as
$$f(p)= \int \G(q^{-1}\cdot p) d \mu q.$$
It follows that $f$ is a nonconstant function which is $C^\infty$ in $\hn \stm C$ and $$\dh f=0 \ \text{on} \ \hn \stm C.$$
Furthermore $f$ is Lipschitz: For $p_1,p_2 \in \hn$ exactly as in the proof of Proposition \ref{stanker}, we obtain
\begin{equation*}
\begin{split}
|f(p_1)-f(p_2)|&=\left| \int \G(q^{-1}\cdot p_1) d \mu q -\int \G(q^{-1}\cdot p_2) d \mu q \right| \\
&\lesssim d(p_1,p_2)\left(\int \frac{1}{d(p_1,q)^{Q-1}}d\mu q+\int \frac{1}{d(p_2,q)^{Q-1}}d\mu q \right) \\
&\lesssim d(p_1,p_2).
\end{split}
\end{equation*} To prove the last inequality let $p \in \hn$, and consider two cases. If $\dist(p,C)>\diam(C)$, 
$$ \int \frac{1}{d(p,q)^{Q-1}}d\mu\leq \frac{ \mu(C)}{\diam(C)^{Q-1}} \lesssim 1.$$
If $\dist(p,C)\leq \diam(C)$, then $C \subset B(p,2\diam(C))$. Let $A=2\diam(C)$, then
\begin{equation*}
\begin{split}
\int \frac{1}{d(p,q)^{Q-1}}d\mu &\leq \sum_{j=0}^\infty \int_{B(p,2^{-j}A)\stm B(p,2^{-(j+1)}A)} \frac{d \mu q}{d(p,q)^{Q-1}}\\
& \leq \sum_{j=0}^\infty \frac{\mu(B(p,2^{-j}A))}{(2^{-(j+1)}A)^{Q-1}}\\
&\leq 2^{Q-1}A^{s-(Q-1)}\sum_{j=0}^\infty (2^{s-(Q-1)})^{-j} \lesssim 1.
\end{split}
\end{equation*}
Since $f \geq 0$ by a Liouville-type theorem for sub-Laplacians, see Theorem 5.8.1 of \cite{blu}, we deduce that $\dh f \not\equiv 0$ on $C$ and hence it is not removable.
\end{proof}

In the following we fix some notation. 

\begin{notation}
\label{singnot}
Recalling (\ref{lapkernels}), (\ref{omega}) and (\ref{vecker}) for a signed Borel measure $\sigma$ set
$$T_\sigma (p):=\int K(q^{-1}\cdot p)d \sigma q, \ \text{whenever it exists},$$
$$T^\e_\sigma (p):=\int_{\hn \stm B(p,\e)} K(q^{-1}\cdot p)d \sigma q$$
and
$$T^\ast_\sigma (p):=\sup_{\e>0}|T^\e_\sigma (p)|.$$
\end{notation}

\begin{rem}
\label{verplane} Vertical hyperplanes of the form $\{(x,t) \in \hn: x\in W, t\in\R\}$, where $W$ is a linear hyperplane 
in $\R^{2n}$, are homogeneous subgroups of $\hn$, that is, they are closed subgroups invariant under the dilations 
$\delta_r$. Their Hausdorff dimension is $Q-1$. 
If $V$ is any such vertical hyperplane and $\sigma$ denotes the $(Q-2)$-dimensional Lebesgue measure on $V$ it follows 
by \cite{Ste}, Theorem 4 p.623 and essentially Corollary 2 p.36, that $T^\ast_\sigma$ is bounded in $L^2(\sigma)$. This 
implies, for example by the methods used in \cite{mpa}, that 
the subsets of vertical hyperplanes of positive measure are not removable for Lipschitz harmonic functions.
\end{rem}

The proof of the following lemma is rather similar to that of Lemma 5.4 in \cite{mpa}.
\begin{lm}
\label{semmes}
Let $\sigma$ be a signed Borel measure in $\hn$ and $A_\sigma$ a positive constant such that
 $|\sigma|(B(p,r)) \leq A_\sigma r^{Q-1}$ for $ p \in \hn, r>0$. Then 
 $$|T^*_\sigma (p)| \leq \|T_\sigma\|_\infty+A_T \ \text{for} \ p \in \hn,$$
 where $A_T$ is a constant depending only on $\sigma$.
\end{lm}
\begin{proof} We can assume that $L=\|T_\s\|_\infty< \infty$. The constants $A_i$ that will appear in the following depend only on $n$ and $\sigma$. For $\e>0$ and $p \in \hn$,
\begin{equation*}
\begin{split}
\frac{1}{\mathcal{L}^{2n+1}(B(p,\e / 2))}&\int_{B(p, \e / 2)} \int_{B(p, \e)}\frac{1}{\|q^{-1} \cdot z\|^{Q-1}}d | \sigma | q d z \\
&\approx  \e^{-Q} \int_{B(p, \e / 2)} \int_{B(p, \e)}\frac{1}{\|q^{-1} \cdot z\|^{Q-1}}d | \sigma | q d z \\
&\leq \int_{B(p, \e)} \e^{-Q} \int_{B(q,2 \e)} \frac{dz}{\|q^{-1} \cdot z\|^{Q-1}}d | \sigma | q \\
&\approx  \e^{1-Q} |\s| (B(p,\e))\leq A_\sigma
\end{split}
\end{equation*}
where we used Fubini and that $$\int_{B(q,2 \e)} \frac{dz}{\|q^{-1} \cdot z\|^{Q-1}} \approx  \e,$$ which is easily checked by summing over the annuli $B(q,2^{1-i} \e) \stm B(q,2^{-i} \e), i=0,1,\dots$.

Now because of the inequality established above we can
choose $z \in B(p, \e / 2)$ with $|T_\sigma(z)| \leq L$ such that 
\begin{equation*}
\int_{B(p, \e)}  |K(q^{-1}\cdot z)|d |\s|q \lesssim \int_{B(p, \e)}\frac{1}{\|q^{-1} \cdot z\|^{Q-1}}d | \sigma | q \leq A_6.
\end{equation*}
Therefore,
\begin{equation*}
\begin{split}
|T_\s^\e(p)-T_\s(z)|&=\left| \int_{\hn \stm B(p, \e)} K(q^{-1}\cdot p)d |\s|q -\int  K(q^{-1}\cdot z)d |\s|q\right| \\
&\leq \int_{\hn \stm B(p, \e)} |K(q^{-1}\cdot p)-K(q^{-1}\cdot z)|d |\s|q+ \int_{B(p, \e)}  |K(q^{-1}\cdot z)|d |\s|q \\
&\leq \int_{\hn \stm B(p, \e)} |K(q^{-1}\cdot p)-K(q^{-1}\cdot z)|d |\s|q +A_6.
\end{split}
\end{equation*}
Furthermore, by Proposition \ref{stanker} (iii), as  $z \in B(p, \e / 2)$,
\begin{equation*}
\begin{split}
\int_{\hn \stm B(p, \e)} &|K(q^{-1}\cdot p)-K(q^{-1}\cdot z)|d |\s|q \\
&\lesssim\int_{\hn \stm B(p, \e)} \max \left\{ \frac{d(p,z)}{d(p,q)^Q},\frac{d(p,z)}{d(z,q)^Q } \right\} d |\s|q\\
&\leq \int_{\hn \stm B(p, \e)} \frac{d(p,z)}{d(p,q)^Q } d |\s|q +\int_{\hn \stm B(z, \e/2)} \frac{d(p,z)}{d(z,q)^Q } d |\s|q 
\end{split}
\end{equation*}
Since
\begin{equation*}
\begin{split}
\int_{\hn \stm B(p, \e)} \frac{d(p,z)}{d(p,q)^Q } d |\s|q &\leq  \frac{\e}{2} \sum_{j=0}^\infty \int_{B(p,2^{j+1} \e) \stm B(p,2^{j} \e) } \frac{1}{d(p,q)^Q } d |\s|q \\
&\leq  \frac{\e}{2} \sum_{j=0}^\infty  \frac{|\s|(B(p,2^{j+1} \e))}{(2^{j} \e)^Q}  \\
&\leq  A_\s\frac{\e}{2} \sum_{j=0}^\infty  \frac{(2^{j+1} \e)^{Q-1}}{(2^{j} \e)^Q} \\
&=A_\s 2^{Q},
\end{split}
\end{equation*}
and in the same way, 
\begin{equation*}
\int_{\hn \stm B(z, \e/2)} \frac{d(p,z)}{d(z,q)^Q } d |\s|q \leq A_\s 2^{Q+1},
\end{equation*}
we deduce that 
\begin{equation*}
\int_{\hn \stm B(p, \e)} |K(q^{-1}\cdot p)-K(q^{-1}\cdot z)|d |\s|q \leq A_7.
\end{equation*}
Therefore
$$|T_\s^\e(p)|\leq|T_\s^\e(p)-T_\s(z)|+|T_\s(z)|\leq A_6 +A_7+L.$$
The lemma is proven.
\end{proof}

\section{$\dh$-removable Cantor sets in $\hn$}

In this section we shall construct a self-similar Cantor set $C$ in $\hn$ which is removable although $0 <\ha^{Q-1}(C) < \infty$. The construction is similar to the one used in \cite{CM} and it is based on ideas of Strichartz in \cite{St}. Notice that in Theorem \ref{remcan} there is one piece $S_0 (C_{r,N})$ of $C_{r,N}$ well separated from the others. This is in order to make the condition of Theorem \ref{unb} easily checkable. It is almost sure that also the more symmetric example used in \cite{CM} would satisfy that condition, but the calculation would become much more complicated.
\begin{df}
\label{similitudes} Let $Q=[0,1]^{2n} \subset \R^{2n},r >0, N \in 2\N$ be such that $r<\frac{1}{N}<\frac{1}{2}$. Let $z_j \in \R^{2n}, j=1,...,N^{2n},$ be distinct
points such that $z_{j,i} \in \{\frac{l}{N}:l=0,1,\cdots,N-1\}$ for all $j=1,\cdots,N^{2n}$ and
$i=1,..,2n$. 

The similarities $\mathcal{S}_{r,N}=\{S_0,..,S_{\frac{1}{2}N^{2n+2}}\}$, depending on the parameters $r$ and $N$, are defined as follows,
\begin{equation*}
\begin{split}
S_0&=\delta_r, \\
S_j&=\t_{(z_{\lfloor j \rfloor_ {N^{2n}}},\frac{1}{2}+\frac{i}{N^2})}\circ\delta_r, \text{ for }i=0,\cdots, \frac{N^2}{2}-1\ \text{and} \ j=iN^{2n}+1,\cdots,(i+1)N^{2n}.\\
\end{split}
\end{equation*}
where $\lfloor j \rfloor_m:=j \mod m$.
\end{df}

\begin{thm}
\label{remcan}
Let $C_{r,N}$ be the self-similar set defined by,
$$C_{r,N}=\bigcup_{j=0}^{\frac{1}{2}N^{2n+2}}S_j (C_{r,N}).$$
Then there exists a set $R \supset C_{r,N}$ such that for all
$j=0,\dots,\frac{1}{2}N^{2n+2}$,
\begin{enumerate}
\item $S_j(R) \subset R$ and
\item the sets $S_j(R)$ are disjoint.
\end{enumerate}
This implies that the sets $S_j(C_{r,N})$ are disjoint for $j=0,\dots,\frac{1}{2}N^{2n+2}$ and
$$0< \mathcal{H}^{a}(C_{r,N})< \infty \text{ with }a=\frac{\log(\frac{1}{2}N^{2n+2}+1)} {\log(\frac{1}{r})}.$$
Furthermore the measure $\mathcal{H}^{a}\lfloor C_{r,N}$ is $a$-AD regular.
\end{thm}
\begin{proof} The proof is almost identical with that of Theorem 4.2 of \cite{CM} but we present it since later we shall need some of its components. Using an idea of Strichartz we show that there exists a
continuous function $\f : Q \ra \R$ such that the set
$$R=\{q \in \hn : q' \in Q \text{ and } \f (q') \leq q_{2n+1} \leq \f (q')+1\}$$
satisfies (i) and (ii).
This will follow if we find some continuous $\f : Q \ra
\R$ which satisfies for all $j=1,\dots,N^{2n}$,
\begin{equation}
\label{ss1} \tau_{(z_j,0)} \delta_r (R) = \{q \in \hn : q' \in Q_j
\text{ and } \f (q') \leq q_{2n+1} \leq \f (q')+r^2\},
\end{equation}
where $Q_j=\t_{(z_j,0)}(\delta_r(Q))$. If (\ref{ss1}) holds then it is readily seen that $R$ satisfies (i). In order to see that $R$ satisfies (ii) as well first notice that (\ref{ss1}) implies that for $j=iN^{2n}+1,\cdots,(i+1)N^{2n}$ and $i=0,\cdots, \frac{N^2}{2}-1$,
\begin{equation}
\label{ss1a}
\begin{split}
S_j(R)&=\t_{(z_{\lfloor j \rfloor_ {N^{2n}}},\frac{1}{2}+\frac{i}{N^2})} \delta_r (R) \\
&= \{q \in \hn : q' \in Q_{\lfloor j \rfloor_ {N^{2n}}} \text{ and } \f (q')+\frac{1}{2}+\frac{i}{N^2} \leq q_{2n+1} \leq \f (q')+\frac{1}{2}+\frac{i}{N^2}+r^2\}.
\end{split}
\end{equation}
Now let $j \neq k \in \{0,\dots, \frac{1}{2}N^{2n}+2\}$ and let $p \in S_j(R)$ and $q \in S_k(R)$. We need to show that $p \neq q$. If $\lfloor j \rfloor_ {N^{2n}} \neq \lfloor k \rfloor_ {N^{2n}}$ then $p' \in Q_{\lfloor j \rfloor_ {N^{2n}}}, \ q' \in Q_{\lfloor k \rfloor_ {N^{2n}}}$, therefore $p' \neq q'$, and so $p \neq q$. If $\lfloor j \rfloor_ {N^{2n}} = \lfloor k \rfloor_ {N^{2n}}$ and $j,k \neq 0$ (the case $jk=0$ is similar and simpler), then there exist $i \neq l \in \{0,\dots, \frac{N^2}{2}-1\}$ such that $j\in\{iN^{2n}+1,\cdots,(i+1)N^{2n}\}$ and $k\in\{lN^{2n}+1,\cdots,(l+1)N^{2n}\}$. Assume without loss of generality that $i>l$. If $p'=q'$ we have by (\ref{ss1a}), since $r <\frac{1}{N} <\frac{1}{2}$,
$$q_{2n+1}\leq \f (q')+\frac{1}{2}+\frac{l}{N^2}+r^2 < \f (q')+\frac{1}{2}+\frac{l+1}{N^2}\leq \f (p')+\frac{1}{2}+\frac{i}{N^2}\leq p_{2n+1}.$$
Hence $p \neq q$ and $S_j(R) \cap S_k(R)= \emptyset$. 

Since
\begin{equation*}
\begin{split}
\tau_{(z_j,0)} \delta_r (R) = \{p \in \hn : p' \in Q_j \text{ and } 
&r^2\f (\frac{p'-z_j}{r})-2\sum_{i=1}^{n}
(z_{j,i}p_{i+n}-z_{j,i+n}p_i) \leq p_{2n+1}\\
&\leq r^2\f (\frac{p'-z_j}{r})-2\sum_{i=1}^{n}
(z_{j,i}p_{i+n}-z_{j,i+n}p_i)+r^2 \},
\end{split}
\end{equation*}
proving (\ref{ss1}) amounts to showing that
\begin{equation}
\label{ss2} \f (w)= r^2\f (\frac{w-z_j}{r})-2\sum_{i=1}^{n}
(z_{j,i}w_{i+n}-z_{j,i+n}w_i) \text{ for }w \in Q_j,j=1,\dots,N^{2n}.
\end{equation}

As usual for any metric space $X$, denote $C(X)=\{f:X\ra \R\text{ and }f \ \text{is continuous}\}$. Let $B=\cup_{j=1}^{N^{2n}} Q_j$ and $L: C(B) \ra C(Q)$ be a linear
extension operator such that
\begin{equation*}
L (f) (x)= f(x) \text { for } x \in B
\end{equation*}
and
$$\|L(f)\|_{\infty}= \|f\|_{\infty}.$$

Since the $Q_j$'s are disjoint the operator $L$ can be defined
simply by taking $\e>0$ small enough and letting

\begin{equation*}
L(f)(x)=
\begin{cases} f(x) & \text{ when }x \in B,
\\
\dfrac{\e - \dist(x,B)}{\e} f (\tilde{x})&\text{ when
}0<\dist(x,B)<\e,
\\
0 & \text{ when } \dist(x,B) \geq \e,
\end{cases}
\end{equation*}
where $\tilde{x} \in B$ and $\dist(x,B)=d(x,\tilde{x}).$

Furthermore define the functions $h: B \ra \R $, $\tilde{f} :B \ra \R$,
$$h(w)=-2\sum_{i=1}^{n}
(z_{j,i}w_{i+n}-z_{j,i+n}w_i) \text{ for }w \in Q_j,$$
$$\tilde{f} (w)= r^2 f (\frac{w-z_j}{r}) \text { for }f \in C(Q),w \in Q_j ,$$
and the operator $T:C(B) \ra C(Q)$ as,
$$T(f)=L(\tilde{f} + h).$$
Then
$$T(f)(w)=r^2 f (\frac{w-z_j}{r})-2\sum_{i=1}^{n}
(z_{j,i}w_{i+n}-z_{j,i+n}w_i) \text{ for }w \in Q_j,$$ and for $f,g
\in C(B)$,
\begin{equation*}
\|Tf-Tg\|_\infty
=\|L(\tilde{f}-\tilde{g})\|_\infty=\|\tilde{f}-\tilde{g}\|_\infty
\leq r^2 \|f-g\|_\infty.
\end{equation*}
Hence $T$ is a contraction and it has a unique fixed point $\f$
which satisfies (\ref{ss2}). The remaining assertions follow from \cite{s}.
\end{proof}
\begin{rem}
\label{positive}
Notice that, by (\ref{ss1a}) in order for all $p \in C_{r,N}\setminus S_0(C_{r,N})$ to satisfy $p_{2n+1}>0$ it suffices to have,
\begin{equation}
\label{pos1}
\f(w) >-\frac{1}{2}\ \text{for all} \ w \in \bigcup_{j=1}^{N^{2n}}Q_j.
\end{equation}
For $w \in Q_j=\prod_{i=1}^{2n}[z_{j,i},z_{j,i}+r],j=1,..,N^{2n}$,
\begin{equation*}
\begin{split}|z_{j,i}w_{i+n}-z_{j,i+n}w_i|&=|z_{j,i}w_{i+n}-w_iw_{i+n}+w_iw_{i+n}-z_{j,i+n}w_i| \\
&\leq |(z_{j,i}-w_i)w_{i+n}|+|w_i(w_{i+n}-z_{j,i+n})| \leq 2r,
\end{split}
\end{equation*}
for all $i=1,...,n.$ 
Hence by (\ref{ss2}) it follows that,
$$|\f(w)| \leq r^2 \|\f\|_\infty + 2\sum_{i=1}^{n}|
z_{j,i}w_{i+n}-z_{j,i+n}w_i| \leq r^2 \|\f\|_\infty+4nr.$$
Therefore
$$\|\f\|_\infty \leq \frac{4nr}{1-r^2} \leq 8nr,$$ 
and (\ref{pos1}) is satisfied if $r < \frac{1}{16 n}.$
\end{rem}

\begin{rem}
\label{q-1dim}
Choose $r_0 < \frac{1}{16 n}$ such that $N_0=\frac{1}{r_0} \in 2 \N$ and consider the self similar sets $C_{r,N_0}, r<r_0$. Then for $r\in(0,r_0)$,
\begin{equation*}
\{\dim C_{r,N_0}:r \in (0,r_0)\} = \left(0,  \frac{\log(\frac{1}{2}N_0^{2n+2}+1)} {\log(N_0)}\right).
\end{equation*}
Furthermore
\begin{equation*}
\frac{\log(\frac{1}{2}N_0^{2n+2}+1)} {\log(N_0)} > \frac{\log(\frac{N_0}{2}) +\log(N_0^{2n+1})} {\log(N_0)}>2n+1.
\end{equation*}
Therefore there exists some $r_{Q-1}  < \frac{1}{N_0}$ such that
$$0< \mathcal{H}^{2n+1}(C_{r_{Q-1},N_0})< \infty.$$
We will denote $C_{r_{Q-1},N_0}$ by $C_{Q-1}$.
\end{rem}

\begin{thm}
\label{unbrem} The Cantor set $C_{Q-1}$ satisfies $0<\mathcal{H}^{Q-1}(C_{Q-1})<\infty$ and is removable.
\end{thm}
\begin{proof}
Suppose that $C_{Q-1}$ is not removable. Then there exists a domain $D\supset C_{Q-1}$ and a Lipschitz function $f:D \ra R$ which is $\dh$-harmonic in $D\stm C_{Q-1}$ but not in $D$. By Theorem \ref{main} there exists a domain $G,C_{Q-1} \subset G \subset D$, a Borel function $h:C \ra \R$ and a $\dh$-harmonic function $H: G \ra \R$ such that 
$$f(p)=\int_{C_{Q-1}} \G(q^{-1}\cdot p)h(q) d \ha^{Q-1}q+H(p) \ \text{for} \ p \in G \stm C_{Q-1}$$
and $\|h\|_{L^\infty (\ha^{Q-1} \lfloor C_{Q-1})}+\|\anah H \|_\infty \lesssim 1$. Let $\s=h \ha^{Q-1} \lfloor C_{Q-1}$. In this case by the left invariance of $\anah$ as in (\ref{im1}) and recalling Notation \ref{singnot}
$$T_\s (p)=\anah f(p)- \anah H(p) \ \text{for all} \ p \in G\stm C_{Q-1}$$ 
which implies that
\begin{equation}
\label{tsb1}
|T_\s (p)| \lesssim 1 \ \text{ for all } \ p \in G\stm C_{Q-1}. 
\end{equation}
Let $\delta=\dist (C_{Q-1}, \hn \stm G )>0$. Then for $p \in \hn \stm G$,
\begin{equation}
\label{tsb2}
|T_\s (p)|\lesssim \int \frac{1}{\|q^{-1} \cdot p\|^{Q-1}} d |\s| q \leq  \frac{|\s|(C_{Q-1})}{\delta^{Q-1}} \lesssim 1. 
\end{equation} 
By (\ref{tsb1}) and (\ref{tsb2}) we deduce that $T_\s \in L^\infty$. Hence, recalling  Theorem \ref{remcan}, since the measure $\ha^{Q-1} \lfloor C_{Q-1}$ is $(Q-1)$-AD regular we can apply Lemma \ref{semmes} and conclude that $T^\ast_\s$ is bounded. Furthermore since $f$ is not harmonic in $C_{Q-1}$, $h \neq 0$ in a set of positive $\ha^{Q-1}$ measure. Therefore there exists a point $p \in C_{Q-1}$ of approximate continuity (with respect to $\ha^{Q-1} \lfloor C_{Q-1}$) of $h$ such that $h(p) \neq 0$. Recalling that $C_{r_{Q-1},N_0}:=C_{Q-1}$ and Definition \ref{similitudes} let $w_k \in \{0,\dots,\frac{1}{2} N^{2n+2}\}^k$ be such that $p \in S_{w_k} (C_{Q-1})$. Then by the approximate continuity of $h$,
$$r^{(1-Q)k}(S_{w_k}^{-1})_\sharp (\s \lfloor S_{w_k}(C_{Q-1})) \rightharpoonup h(p)\ha^{Q-1} \lfloor C_{Q-1} \text{ as }  k \ra \infty,$$
and the boundedness of $T^\ast_\s$ implies that $T^\ast_{\ha^{Q-1} \lfloor C_{Q-1}}$ is bounded. To see this let $z \in \hn \stm (C_{Q-1}\cup \bigcup_{k=1}^\infty S^{-1}_{w_k}(C_{Q-1}))$. If $\dist(z,C_{Q-1})> \frac{\alpha_{C_{Q-1}}}{2}\diam(C_{Q-1})$, then
\begin{equation}
|T_{\ha^{Q-1} \lfloor C_{Q-1}}(z)|\lesssim 1. 
\end{equation} 
Therefore we can assume that $\dist(z,C_{Q-1})\leq \frac{\alpha_{C_{Q-1}}}{2}\diam(C_{Q-1})$. Recalling Remark \ref{q-1dim} this implies that for any $w \in \mathcal{I}$,
\begin{equation}
\label{dissw}
\begin{split}
\dist (S_w (z), S_w(C_{Q-1}))&=r_{Q-1}^{|w|}\dist(z,C_{Q-1}) \\
&\leq r_{Q-1}^{|w|}\frac{\alpha_{C_{Q-1}}}{2} \diam (C_{Q-1})=\frac{\alpha_{C_{Q-1}}}{2} \diam (S_w(C_{Q-1})).
\end{split}
\end{equation}
Notice that the homogeneity of $K$  implies that $K(S^{-1}_{w_k}(q)^{-1}\cdot z)=r^{(Q-1)k}K(q^{-1}\cdot S_{w_k}(z))$ as in the proof of Theorem \ref{unb}. Therefore by (\ref{homome}), 
\begin{equation*}
\begin{split}
h(p)T_{\ha^{Q-1} \lfloor C_{Q-1}}(z)&=\lim_{k \ra \infty} r^{(1-Q)k}\int K(q^{-1}\cdot z)d(S_{w_k}^{-1})_\sharp (\s \lfloor S_{w_k}(C_{Q-1}))q \\
&=\lim_{k \ra \infty} r^{(1-Q)k}\int_{S_{w_k}(C_{Q-1})} K(S^{-1}_{w_k}(q)^{-1}\cdot z) d \s q\\
&=\lim_{k \ra \infty} \int_{S_{w_k}(C_{Q-1})} K(q^{-1}\cdot S_{w_k}(z)) d \s q \\
&=\lim_{k \ra \infty} \left(\int_{C_{Q-1}} K(q^{-1}\cdot S_{w_k}(z)) d \s q-\int_{C_{Q-1} \stm S_{w_k}(C_{Q-1})} K(q^{-1}\cdot S_{w_k}(z)) d \s q\right).
\end{split}
\end{equation*}
Since $z \notin \bigcup_{k=1}^\infty S^{-1}_{w_k}(C_{Q-1})$,
\begin{equation*}
\left|\int_{C_{Q-1}} K(q^{-1}\cdot S_{w_k}(z)) d \s q\right| \leq \|T^\ast_\s\|_\infty.
\end{equation*}
Furthermore by Lemma \ref{compop} and (\ref{dissw}) we get that,
\begin{equation*}
\left|\int_{C_{Q-1} \stm S_{w_k}(C_{Q-1})} K(q^{-1}\cdot S_{w_k}(z)) d \s q\right| \leq 2 \|T^\ast_\s\|_\infty+A_{C_{Q-1}}.
\end{equation*}
Therefore,
$$|h(p) T_{\ha^{Q-1} \lfloor C_{Q-1}}(z)| \leq 3 \|T^\ast_\s\|_\infty+A_{C_{Q-1}},$$
and since $$\mathcal{L}^{2n+1}\left(C_{Q-1}\cup \bigcup_{k=1}^\infty S^{-1}_{w_k}(C_{Q-1})\right)=0$$ we get that $T_{\ha^{Q-1} \lfloor C_{Q-1}} \in L^\infty$. Hence by Lemma \ref{semmes} $T^\ast_{\ha^{Q-1} \lfloor C_{Q-1}}$ is bounded.

On the other hand notice that
\begin{equation}
\begin{split}
\label{lastestim}
\int_{C_{Q-1} \stm S_0 (C_{Q-1})}& K_{1+n} (q^{-1} ) d \ha^{Q-1} q \\
&=\int_{C_{Q-1} \stm S_0 (C_{Q-1})} c_Q\frac{(-q_{1+n})|q'|^2-(-q_{1})(-q_{2n+1})}{\|q\|^{Q+2}}d \ha^{Q-1} q\\
&=-\int_{C_{Q-1} \stm S_0 (C_{Q-1})} c_Q\frac{q_{1+n}|q'|^2+q_{1}q_{2n+1}}{\|q\|^{Q+2}}d \ha^{Q-1} q. 
\end{split}
\end{equation}
Recalling Definition \ref{similitudes} for $q \in C_{Q-1} \stm S_0 (C_{Q-1})$, $q_{1+n},q_1 \in [0,1]\stm[0,r_{Q-1}]$ and by Remark \ref{positive} we also have that $q_{2n+1}>0$. Hence  $q_{1+n}|q'|^2+q_{1}q_{2n+1}>0$ for $q \in C_{Q-1} \stm S_0 (C_{Q-1})$ and by (\ref{lastestim})
$$ \int_{C_{Q-1} \stm S_0 (C_{Q-1})} K_{1+n} (q^{-1} ) d \ha^{Q-1} q \neq 0.$$
Therefore, by Theorem \ref{unb} (recall the definition of fixed points of a family similarities given before it), since $0$ is a fixed point for $\mathcal{S}_{r_{Q-1},N_0}$, more precisely $S_0(0)=0$, $T^\ast_{K_{n+1}} (\ha^{Q-1} \lfloor C_{Q-1})$ and hence $T^\ast_{\ha^{Q-1} \lfloor C_{Q-1}}$ is unbounded. We have reached a contradiction and the theorem is 
proven.
\end{proof}
   
\section{Concluding comments and questions}
Here we shall discuss some questions that are left unanswered, or even not considered at all, so far.

What $(Q-1)$-dimensional subsets of $\hn$ are not removable? The proof of Theorem \ref{unbrem} uses the special structure of 
$C_{Q-1}$ only at the end to check the condition of Theorem \ref{unb}. It is quite likely that the cases of self-similar 
sets where this condition fails are quite exceptional, but checking it could be technically very complicated. In our 
case we set up the example so that the integrand doesn't change sign, but even for the sets considered in \cite{CM} one would 
need to compare carefully the positive and negative contributions. Note also that there are actually infinitely many 
sufficient conditions for the unboundedness in Theorem \ref{unb} corresponding to the dense set of fixed points.

The related question is on what $(Q-1)$-dimensional subsets of $\hn$ the singular integral operator related to the kernel 
$K=\anah \Gamma$ can be $L^2$-bounded. Or on what $m$-dimensional subsets of $\hn$ the singular integral operators related 
to appropriate $m$-homogeneous kernels can be $L^2$-bounded. As mentioned in the introduction essentially complete results 
are only known for the Cauchy kernel in the complex plane (or also for the Riesz kernel $|x|^{-2}x$ in $\R^n$). For 
$m$-dimensional Ahlfors-David-regular sets $E$ and $m$-homogeneous Riesz kernels in $\R^n$ we know that the 
$L^2$-boundedness 
implies that $m$ must be an integer, \cite{V}, and $E$ must be well approximated by $m$-planes almost everywhere at 
some arbitrarily small scales, \cite{mpa}, \cite{M4}. 
Similar results were proved for Riesz-type kernels in $\hn$ in \cite{CM}. A property of these 
kernels $R$ that was crucial for the proofs is that $R(x)=-R(y)$ if and only if $x=-y$. Obtaining similar results 
even for the simple kernel  $z^3 /|z|^4$ in $\C$ does not seem to be trivial, and far less for the kernel $K=\anah \Gamma$
in $\hn$. 

We have not studied here the converse: what regularity properties of the underlying sets guarantee the $L^2$-boundedness 
of the singular integral operators and the non-removability of such sets? 
In $\R^n$ this is well understood by the results of 
David and Semmes, see \cite{DS}. They have proved that a large class of singular integral operators are $L^2$-bounded 
on uniformly rectifiable sets (which include Lipschitz graphs, for example), and this is essentially the best one can say. 
It follows that compact subsets $C$ of $(n-1)$-dimensional uniformly rectifiable sets with  $\ha^{n-1}(C)>0$ 
are not removable for Lipschitz harmonic functions in $\R^n$. In $\hn$ it would be natural to start asking what smoothness properties of surfaces guarantee the $L^2$-boundedness of various singular integral operators? An extensive study of  surfaces in $\hn$ is performed in \cite{fss1}. The horizontal surfaces of \cite{fss1}, being essentially Euclidean, should be easier to handle than the vertical ones. As in Remark \ref{verplane}, the 
general results in \cite{Ste} can be used in vertical subgroups. In particular, 
 the subsets of positive measure of vertical hyperplanes are not removable for Lipschitz harmonic functions.

Our final comment is actually irrelevant for this paper, but we would like to straighten one item of \cite{CM}. As 
observed by Enrico Le Donne, the proof of Lemma 2.11 in \cite{CM} is too complicated and the question stated in 
Remark 2.12 has a positive answer. This was also proved and used in a different setting in [AKL].


\begin{thebibliography}{CMM}
\bibitem[AFP]{afp} L. Ambrosio, N. Fusco and D. Pallara, \emph{Functions of Bounded Variations and Free Discontinuity Problems}, Oxford Mathematical Monographs. The Clarendon Press, Oxford University Press, New York.

\bibitem[AKL]{AKL} L. Ambrosio, B. Kleiner and E. Le Donne, \emph{Rectifiability of sets of finite perimeter in Carnot
groups: existence of a tangent hyperplane}, J. Geom. Anal. 19 (2009), no. 3, 509-540.


\bibitem[BLU]{blu} A. Bonfiglioli, E. Lanconelli and F. Uguzzoni, \emph{Stratified Lie Groups and Potential Theory for Their Sub-Laplacians}, Springer Monographs in Mathematics 2007.

\bibitem[Ca]{ca} L. Carleson, \emph{Selected Problems on Exceptional Sets}, D. Van Nostrand Company, Inc., Princeton, NJ (1966).

\bibitem[CDPT]{cap} L. Capogna, D. Danielli, S. D. Pauls and J. T. Tyson, \emph{An Introduction
to the Heisenberg Group and the Sub-Riemannian Isoperimetric Problem}, Birkh\"auser 2007.


\bibitem[C]{C} V. Chousionis, \emph{Singular integrals on Sierpinski gaskets}, Publ. Mat. 53  (2009),  no. 1, 245--256.

\bibitem[CM]{CM} V. Chousionis and P. Mattila, \emph{Singular integrals on Ahlfors-David subsets of the Heisenberg group}, J. Geom. Anal. 21 (2011), no. 1 , 56--77.

\bibitem[D1]{d} G. David,\emph{Unrectifiable 1-sets have vanishing analytic capacity}, Rev. Math. Iberoam. 14 (1998) 269--479.

\bibitem[D2]{dc} G.David, \emph{Des int\'{e}grales singuli\`{e}res
born\'{e}es sur un ensemble de Cantor},  C. R. Acad. Sci. Paris Sr.
I Math. 332  (2001),  no. 5, 391--396.

\bibitem[DM]{dm} G. David and P. Mattila, \emph{Removable sets for Lipschitz harmonic functions in
the plane}, Rev. Mat. Iberoamericana 16 (2000), no. 1, 137--215.


\bibitem[DS]{DS} G. David and S. Semmes, \emph{Analysis of and on Uniformly
Rectifiable Sets}, Mathematical Surveys and Monographs, 38. American
Mathematical Society, Providence, RI, (1993).


\bibitem[ENV]{env} V. Eiderman, F. Nazarov and A. Volberg, \emph{Vector-valued Riesz potentials: Cartan type estimates and related capacities}, Proc. London Math. Soc. (2010) 101 (3): 727--758.

\bibitem[F]{Fe} H. Federer, \emph {Geometric Measure Theory}, Springer-Verlag New York Inc., New York 1969.

\bibitem[Fo]{fol} G. B. Folland, \emph{Subelliptic estimates and function spaces on nilpotent Lie groups}, Ark. Mat.  13  (1975), no. 2, 161--207. 

\bibitem[FSSC1]{fss} B. Franchi, R. Serapioni and F. Serra Cassano, \emph {Rectifiability and perimeter in the Heisenberg group}, Math. Ann. 321 (2001), 479--531.

\bibitem[FSSC2]{fss1} B. Franchi, R. Serapioni and F. Serra Cassano, \emph {Regular submanifolds, graphs and area formula in Heisenberg Groups}, Advances in Math, 211, 1, (2007), 152--203.


\bibitem[G1]{g} J. Garnett, \emph{Positive length but zero analytic capacity}, Proc. Amer. Math. Soc. 21 (1970), 696--699.

\bibitem[G2]{gb} J. Garnett, \emph{Bounded Analytic Functions}, Academic Press, New York, 1981.

\bibitem[GN]{gn} N. Garofalo and D-M. Nhieu, \emph {Lipschitz continuity, global smooth approximations and extension theorems for Sobolev functions in Carnot-Carath\'eodory spaces}, Journal d'Analyse Math., 74 (1998), 67--97.

\bibitem[GV]{gv} J. Garnett and J. Verdera, \emph{Analytic capacity, bilipschitz maps and Cantor sets}, Math. Res. Lett. 10 (2003), no. 4, 515--522.

\bibitem[Gr]{Gr} L. Grafakos, \emph {Classical and Modern Fourier Analysis}, Pearson Education, Inc., Upper Saddle River, NJ, 2004.


\bibitem[H]{h} P. Huovinen, \emph{Singular integrals and rectifiability of measures in the plane}, Ann. Acad. Sci. Fenn. Math. Diss. No. 109 (1997)

 \bibitem[I1]{i} L. D. Ivanov, \emph{Variations of sets and functions}, Nauka, Moscow, 1975 (Russian). MR 57:16498.


\bibitem[I2]{i2} L. D. Ivanov, \emph{On sets of analytic capacity zero}, in Linear and Complex Analysis Problem Book 3 (part II), Lectures Notes in Mathematics 1574, Springer-Verlag, Berlin, 1994, pp. 150--153.

\bibitem[J]{j} P. Jones, \emph{Square functions, Cauchy integrals, analytic capacity, and harmonic measure}, in: "Harmonic Analysis and Partial Differential Equations", (El Escorial, 1987), Lectures Notes in Mathematics 1384, Springer, Berlin, 1989, pp. 24-68. 

\bibitem[JM]{jm} P. Jones and T. Murai, \emph{Positive analytic capacity but zero Buffon needle probability}, Pacific J. Math. 133  (1988),  no. 1, 99--114.

\bibitem[L]{l} A. Lorent,  \emph{A generalised conical density theorem for unrectifiable sets}, Ann. Acad. Sci. Fenn. Math.  28  (2003), no. 2, 415--431.

\bibitem[MT]{mt} J. Mateu and X. Tolsa, \emph{Riesz transforms and harmonic Lip1-capacity in Cantor sets}, Proc. London Math. Soc. 89(3) (2004), 676--696.

\bibitem[MTV1]{mtv} J. Mateu, X. Tolsa and J. Verdera, \emph{The planar Cantor sets of zero analytic capacity and the local T(b)-Theorem}, J. Amer. Math. Soc. 16 (2003), 19--28.

\bibitem[MTV2]{mtv2} J. Mateu, X. Tolsa and J. Verdera, \emph{On the semiadditivity of analytic capacity and planar Cantor sets}, Harmonic analysis at Mount Holyoke (South Hadley, MA, 2001),  259--278, Contemp. Math., 320, Amer. Math. Soc., Providence, RI, 2003.


\bibitem[M1]{M} P. Mattila, \emph{Geometry of Sets and Measures in Euclidean Spaces},
Cambridge University Press, (1995).

\bibitem[M2]{M2} P. Mattila, \emph{On the analytic capacity and curvature of some Cantor sets with non-$\sigma$-finite length}, Publ. Mat. 40  (1996), no. 1, 195--204.

\bibitem[M3]{M4} P. Mattila, \emph{Singular integrals and rectifiability}, Proceedings of the 6th International Conference on Harmonic Analysis and Partial Differential Equations (El Escorial, 2000).  Publ. Mat.  2002,  Vol. Extra, 199--208.

\bibitem[M4]{ms} P. Mattila, \emph{Measures with unique tangent measures in metric groups}, Math.
Scand. 97 (2005), 298--398.

\bibitem[M5]{mro} P. Mattila, \emph{ Removability, singular integrals and rectifiability}, Rev. Roumaine Math. Pures Appl.  54  (2009),  no. 5-6, 483--491.

\bibitem[MMV]{MMV} P. Mattila, M. S. Melnikov and J. Verdera,
\emph{The Cauchy integral, analytic capacity, and uniform rectifiability},
Ann. of Math. (2)  144  (1996),  no. 1, 127--136.

\bibitem[MPa]{mpa} P. Mattila and P. V. Paramonov,
\emph{On geometric properties of harmonic Lip1-capacity}, Pacific J. Math., 171:2 (1995), 469--490. 

\bibitem[Me]{me} M. S. Melnikov, \emph{Analytic capacity: discrete approach and curvature of a measure}, Sbornik: Mathematics 186(6) (1995), 827--846.

\bibitem[MeV]{mv} M. S. Melnikov and J. Verdera, \emph{A geometric proof of the $L^2$ boundedness of
the Cauchy integral on Lipschitz graphs}, Internat. Math. Res. Notices (1995),
325--331.

\bibitem[Pa]{Pa} H. Pajot, \emph{Analytic capacity, rectifiability, Menger curvature and the Cauchy integral}, Lecture Notes in Mathematics, Vol. 1799. Springer-Verlag, Berlin.

\bibitem[Pan]{pan} P. Pansu, \emph{M\'etriques de Carnot-Carathodory et quasiisom\'etries des espaces sym\'etriques de rang un}, Ann. of Math. (2)  129  (1989),  no. 1, 1--60.


\bibitem[S]{s} A. Schief, \emph{Self-similar sets in complete metric spaces}, 
Proc. Amer. Math. Soc. 124 (1996), 481-490.

\bibitem[St]{Ste} E. M. Stein, \emph{Harmonic Analysis: Real-variable Methods,
Orthogonality and Oscillatory Integrals}, -Princeton University
Press, Princeton New Jersey, (1993).


\bibitem[Str]{St} R. Strichartz, \emph{Self-similarity on nilpotent Lie groups}, Contemp. Math. 140, (1992) 123--157. 


\bibitem[T1]{To} X. Tolsa, \emph{Painlev\'e's problem and the semiadditivity of analytic capacity}, Acta Math. 190:1 (2003), 105--149.

\bibitem[T2]{tre} X. Tolsa, \emph{Analytic capacity, rectifiability, and the Cauchy integral}, Contemp. International Congress of Mathematicians. Vol. II,  1505--1527, Eur. Math. Soc., Z\"urich, 2006.

\bibitem[T3]{tr} X. Tolsa, \emph{Uniform rectifiability, Calder\'on-Zygmund operators with odd kernel, and quasiorthogonality}, Proc. Lond. Math. Soc. (3)  98  (2009), no. 2, 393--426.

\bibitem[T4]{t2} X.Tolsa, \emph{Calder\'on-Zygmund capacities and Wolff potentials on Cantor sets}, to appear in J. Geom. Anal.

\bibitem[Vi]{V} M. Vihtil\"a, \emph{The boundedness of Riesz $s$-transforms of measures in $\Rn$}, 
Proc. Amer. Math. Soc. 124 (1996), 481--490.

\bibitem[Vo]{vo} A. Volberg, \emph{Calder\'on-Zygmund capacities and operators on nonhomogeneous spaces}, CBMS Regional Conference Series in Mathematics 100. American Mathematical Society, Providence, RI, 2003.


\end{thebibliography}
\end{document}